\documentclass[12pt,a4paper]{amsart}

\usepackage[colorlinks=true, urlcolor=black, citecolor=black, linkcolor=black, hyperfootnotes=true]{hyperref}
\usepackage[capitalise]{cleveref}

\usepackage{amsmath,amsthm,amssymb,latexsym,a4wide,tikz,multicol,tikz-cd}
\usepackage{arydshln,multirow}
\usepackage{tikz-qtree,tikz-qtree-compat}
\usepackage{mathtools,stmaryrd}
\tikzset{font=\small}
\usepackage{enumerate}
\usetikzlibrary{matrix,arrows}
\usetikzlibrary{positioning}
\usetikzlibrary{cd}
\usepackage{cite}
\usepackage[active]{srcltx}

\newtheorem{theorem}{Theorem} [section]
\newtheorem{lemma}[theorem]{Lemma}
\newtheorem{corollary}[theorem]{Corollary}
\newtheorem{proposition}[theorem]{Proposition}
\theoremstyle{definition}
\newtheorem{definition}[theorem]{Definition}
\newtheorem{remark}[theorem]{Remark}

\numberwithin{equation}{section}
\usepackage{enumitem}
\setlist[enumerate]{itemsep=0mm}
\setlist[itemize]{itemsep=0mm}
\let\OLDthebibliography\thebibliography
\renewcommand\thebibliography[1]{
  \OLDthebibliography{#1}
  \setlength{\parskip}{0pt}
  \setlength{\itemsep}{0pt plus 0.3ex}
}

\newcommand{\UU}{\mathcal{U}}

\begin{document}

\title[Boolean-Cartan Restriction Semigroups]{Steinberg Algebras of Ample Semicategories and their Boolean-Cartan Restriction Semigroups}

\date{}
\author{Tristan Bice}
\address{T. Bice: Institute of Mathematics of the Czech Academy of Sciences, \v{Z}itn\'a 25, Prague}
\email{bice@math.cas.cz}

\author{Malcolm Jones}
\address{M. Jones: Institute of Mathematics, Physics and Mechanics, Jadranska ulica 19, SI-1000 Ljubljana, Slovenia}
\email{malcolm.jones@imfm.si}

\author{Ganna Kudryavtseva}
\address{G. Kudryavtseva: University of Ljubljana,
Faculty of Mathematics and Physics, Jadranska ulica 19, SI-1000 Ljubljana, Slovenia / Institute of Mathematics, Physics and Mechanics, Jadranska ulica 19, SI-1000 Ljubljana, Slovenia}
\email{ganna.kudryavtseva@fmf.uni-lj.si}

\thanks{The second named author was supported by the Slovenian Research and Innovation Agency grant P1-0288. The third named author was partially supported by the Slovenian Research and Innovation Agency grants P1-0288 and J1-60025. The authors would like to thank Lisa Orloff Clark for helpful comments.}

\begin{abstract}
We extend the construction of Steinberg algebras of ample groupoids to \'etale semicategories.  We also relate ample semicategories to Boolean restriction semigroups via a representation result extending previously known results for categories.  Furthermore, we prove a reconstruction result which characterises an abstract algebra $A$ with a certain Cartan-like restriction subsemigroup $B$ (subject to conditions resembling those defining quasi-Cartan pairs) as the Steinberg algebra of the ultrafilter groupoid of $B$.  In this way we obtain a twist-free extension of previous Steinberg algebra reconstruction results.

\vspace{0.3cm}


\vspace{0.3cm}

\end{abstract}

\maketitle

\setcounter{tocdepth}{1}
\tableofcontents 

\section{Introduction}

\subsection{Background}

This paper contributes to a broad area of research on the reconstruction of Steinberg algebras and \'etale groupoid $C^*$-algebras, which has been mainly developed in the following two interconnected, yet distinct, directions.

The first direction concerns the recovery of a groupoid from certain algebras and semigroups of functions defined upon it and encompasses the reconstruction results for groupoid $C^*$-algebras \cite{Kumjian1986,Renault2008} and their subsemigroups \cite{BiceClark2021}, Leavitt path algebras \cite{BrownClarkHuef2017} and Steinberg algebras \cite{AraBosaHazratSims2017,CarlsenRout2018}. Extending the latter results, Steinberg showed in \cite[Theorem 5.7]{Ste19} that, under suitable hypotheses, a diagonal-preserving isomorphism between Steinberg algebras induces an isomorphism of the underlying groupoids.

The other direction stems from the work of Kumjian \cite{Kumjian1986} and Renault \cite{Renault2008} who, building on ideas of Feldman and Moore \cite{FeldmanMoore1977} in the setting of von Neumann algebras, introduced the notion of a Cartan pair $(A,B)$, where $A$ is an abstract $C^*$-algebra and $B$ is a maximal abelian subalgebra that is regular, contains an approximate unit for $A$, and admits a faithful conditional expectation from $A$ onto $B$.  Renault proved in \cite[Theorem 5.6]{Renault2008} that, for every Cartan pair $(A,B)$, $A$ is isomorphic to the reduced twisted $C^*$-algebra $C^*_r(G(B), \Sigma(B))$, where $G(B)$ is the Weyl groupoid of $(A,B)$ and $\Sigma(B)$ is the associated Weyl twist.  By axiomatising the relevant properties of normaliser semigroups, this was extended to non-reduced $C^*$-algebras of even non-effective groupoids in \cite{BiceClarkYingFenMcCormick2025}.

Armstrong et al. \cite{9a23} also proved analogous results characterising twisted Steinberg algebras (as introduced in \cite{ACCLMR22}) of effective ample groupoids by introducing quasi-Cartan pairs, a purely algebraic analogue of Renault's Cartan pairs.  Indeed, it is proved in \cite[Theorem 6.6]{9a23} that, given a quasi-Cartan pair $(A,B)$, the algebra $A$ is isomorphic to the twisted Steinberg algebra $A_R(G,\Sigma)$ where $G$ and $\Sigma$ are groupoids constructed using the normaliser semigroup $N_A(B)$ of $B$ in $A$.

Recognising an abstract algebra or an abstract $C^*$-algebra as a twisted Steinberg algebra  or as  a twisted groupoid $C^*$-algebra allows one to apply the well-developed topological and dynamical techniques available for the study of the algebra in question. This also applies to the non-twisted setting. For example, once Leavitt path algebras and graph $C^*$-algebras were recognized to be groupoid algebras \cite{CFST14,KumjianPaskRaeburnRenault1997}, the rapid and unified development of their theory was enabled. Given this, we found it surprising that the question of recognizing an abstract $R$-algebra $A$ as a Steinberg algebra has not been previously considered in the non-twisted setting. 
One of the aims of this paper is to fill in this gap in the literature.

The second aim of this paper is to extend the level of generality of the reconstruction results by working with Steinberg algebras of ample Hausdorff semicategories rather than of ample groupoids. To achieve this, we first define Steinberg algebras of ample Hausdorff semicategories, inspired by the recent notion of a Steinberg algebra of an ample category introduced by Kudryavtseva in \cite{Kudryavtseva2025} and independently as Orloff algebras by Ruiz Campos in \cite{Cam25}.

It is worth emphasizing that, because the setting of semicategories lacks the left-right symmetry of groupoids, the strategy adopted in \cite{Renault2008,Ste19,9a23} of working with the normaliser semigroup $N_A(B)$ does not generalise to the present setting.  We therefore adopt a different strategy by considering an algebra $A$ together with a certain Boolean restriction subsemigroup $B$ which plays a role analogous to that of the normaliser semigroup $N_A(B)$ in earlier work, much like the special semigroups considered in \cite{Bice2023} and \cite{BiceClarkYingFenMcCormick2025}.  In particular, these Boolean-Cartan semigroups have a meet operation which extends linearly to $A$ and which, in groupoid setting, can be reconstructed from the usual conditional expectation on the Steinberg algebra.  Utilising this, we prove that $A$ is in fact isomorphic to the algebra of the ultrafilter semicategory of $B$. Another important distinction between our work and earlier work is that we do not need any local bisection hypothesis (which generalizes the classical no non-trivial units hypothesis required to reconstruct a group from a group ring).

Before proving our reconstruction results, we introduce and explore the notion of an \'etale semicategory and its Steinberg algebra. Furthermore, we work with semi-Boolean restriction semigroups, whose projections form a semi-Boolean algebra rather than a generalized Boolean algebra as for the pre-Boolean restriction semigroups of \cite{Kudryavtseva2025}.\footnote{Our semi-Boolean restriction semigroups are not, however, proper generalizations of pre-Boolean restriction semigroups of \cite{Kudryavtseva2025}, as the semigroups there need not have meets.}  We prove that such a semi-Boolean restriction semigroup $B$ can be represented as a restriction semigroup formed by a basis of compact and open local sections of its groupoid of ultrafilters ${\mathcal U}(B)$. This paves the way for developing a non-commutative duality theory for more general semi-Boolean restriction semigroups without meets, which we aim to develop in future work.

\subsection{Outline}

The structure of the paper is as follows. In Section~\ref{SemiBooleanAlgebras}, we provide preliminaries on semi-Boolean algebras, proving that for such an algebra $B$ its ultrafilter space ${\mathcal U}(B)$ is a locally compact Stone space and that $B$ is isomorphic to a semi-Boolean algebra formed by a basis of compact open sets of  ${\mathcal U}(B)$ -- see \Cref{LocallyCompact,StoneRep}. In Section~\ref{s:3}, we recall the notion of a semigroupoid and study those with a distinguished set of projections.  We then provide preliminaries on restriction semigroups and finally single out a subclass of semigroupoids which we call semicategories.  Further, in Section~\ref{s:4} we define and study \'etale semicategories.  We prove in Proposition \ref{OpenProduct} that the product map in an \'etale semicategory is an open map.  Proposition \ref{SemiSlices} provides a key connection of \'etale semicategories with restriction semigroups showing that the open sections of any \'etale semicategory $S$ form a restriction semigroup relative to the open subsets of $S^0$ and, furthermore, if $S^2$ is closed in $S\times S$, then the compact open sections of $S$ form a restriction semigroup relative to the compact open subsets of $S^0$.

In Section \ref{s:5}, we generalise Steinberg algebras to any Hausdorff \'etale semicategory $S$ where $S^2$ is closed, see \Cref{KSalg}.  We work at the level of generality where the scalars belong to a not necessarily commutative ring $K$ with unit, meaning that by a $K$-algebra we mean a $K$-bimodule $A$ with an associative multiplication which is distributive with respect to addition.  
Section \ref{s:6} develops a topological representation theory of semi-Boolean restriction semigroups as compact open sections of its semicategory of ultrafilters (see Theorem \ref{AmpleSemi}).  Here we also characterise when this semicategory becomes a category or a groupoid -- see Propositions \ref{UltraCategory} and \ref{UltraGroupoid}.

In Section \ref{s:BCrs} we present our central definition of a Boolean-Cartan restriction semigroup $B$ in a $K$-algebra $A$ (Definition \ref{def:BC}) and prove in Theorem \ref{SteinbergAlgebraRepresentation} the reconstruction result that an algebra $A$ containing a Boolean-Cartan restriction semigroup $B$ is isomorphic to the Steinberg algebra of the ultrafilter semicategory of $B$. In Section \ref{s:8} we show that a restriction semigroup $B$ is in fact a Boolean-Cartan restriction semigroup in a certain quotient $KB_0/I$ of the contracted semigroup algebra $KB_0$. Applying the reconstruction result of the previous section, we conclude that if $B$ is a Boolean-Cartan restriction semigroup in an algebra $A$, then $A$ must be isomorphic to $KB_0/I$. We conclude the paper with Section \ref{s:9} which relates our Boolean-Cartan semigroups to various commutative subalgebras studied in \cite{9a23}, where the ring $K$ is assumed to be commutative and indecomposable.

\section{Semi-Boolean Algebras}\label{SemiBooleanAlgebras}

Let us start with some general notation.  Given any relation $R\subseteq T\times U$ and $t\in T$, let
\[t^R=\{u\in U\mathrel{|}t\mathrel{R}u\}.\]
More generally, for any $S\subseteq T$, let
\[S^R=\bigcup_{s\in S}s^R.\]

So if $B$ is a poset and $S\subseteq B$ then, applying the notation given above, $S^\leq$ denotes the upwards-closure while $S^\geq$ denotes the downwards closure of $S$.  In particular, $b^\leq$ is the \emph{principal filter} while $b^\geq$ is the \emph{principal ideal} defined by any given $b\in B$.  A \emph{meet} of $a,b\in B$ is a (necessarily unique) element $a\wedge b\in B$ such that
\[\tag{Meet}a\wedge b=\max(a^\geq\cap b^\geq).\]
A \emph{meet-semilattice} is a poset in which all pairs have meets.
Likewise, a \emph{join} of $a,b\in B$ is a (necessarily unique) element $a\vee b\in B$ such that
\[\tag{Join}a\vee b=\min(a^\leq\cap b^\leq)\]
and a \emph{join-semilattice} is a poset in which all pairs have joins.  A \emph{lattice} is a poset in which all pairs have both meets and joins.  A lattice $B$ is \emph{distributive} if, for all $a,b,c\in B$,
\[\tag{Distributivity}a\wedge(b\vee c)=(a\wedge b)\vee(a\wedge c).\]

If a poset $B$ has a minimum $0$, we call $a,b\in B$ \emph{disjoint} if $a\wedge b=0$.  We denote this by $\top$, i.e.
\[a\mathrel{\top}b\qquad\Leftrightarrow\qquad a\wedge b=0.\]
A \emph{pseudocomplement} of $b\in B$ is a (necessarily unique) element $b'\in B$ such that
\[\tag{Pseudocomplement}b'=\max(b^\top).\]
For all $a, b \in B$, if $a \leq b$, then $b^\top \subseteq a^\top$. If further $a$ and $b$ have pseudocomplements, then $b' \in b^\top \subseteq a^\top$, and so $b' \leq a'$.
We call $B$ \emph{pseudocomplemented} if every $b\in B$ has a pseudocomplement.  If $B$ also has a maximum $1$ then $a,b\in B$ are \emph{complements} when $a\wedge b=0$ and $a\vee b=1$.  We call $B$ \emph{complemented} if every $b\in B$ has a complement.  A \emph{Boolean algebra} is a complemented distributive lattice.  If $a$ and $b$ are complements in a Boolean algebra $B$ and $c \in b^\top$, then $c = 1 \wedge c = (b \vee a) \wedge c = (b \wedge c) \vee (a \wedge c) = 0 \vee (a \wedge c) = a \wedge c$. Hence, $a = \max(b^\top) = b'$. 
That is, complements in a Boolean algebra $B$ are pseudocomplements and hence $b''=b$, for all $b\in B$.  In the presence of meets or joins, this in fact characterises Boolean algebras, as first noted by Byrne in \cite{Byrne1946}.

\begin{proposition}\label{Byrne}
    If a poset $B$ is pseudocomplemented with $b''=b$, for all $b\in B$, and $B$ is either a meet-semilattice or a join-semilattice then $B$ is in fact a Boolean algebra.
\end{proposition}

\begin{proof}
As $b''=b$ and $b'\leq a'$ whenever $a\leq b$, the pseudocomplement operation is an anti-isomorphism on $B$ and thus maps meets/joins to joins/meets.  As $B$ is a meet/join-semilattice with minimum $0$, it follows that $B$ is also a join/meet-semilattice with maximum $1=0'$.  Moreover, as $b\wedge b'=0$ it follows that $b'\vee b=b'\vee b''=0'=1$, for all $b\in B$, i.e.~pseudocomplements are complements.  To see that $B$ is also distributive, say $a\leq b\vee c$ and let $d=a\wedge((a\wedge b)\vee(a\wedge c))'$.  Then $d\wedge b=d\wedge a\wedge b=0$ and $d\wedge c=d\wedge a\wedge c=0$ so $d\leq b'\wedge c'=(b\vee c)'$ even though $d\leq a\leq b\vee c$.  Thus $d=0$ so $a\leq((a\wedge b)\vee(a\wedge c))''=(a\wedge b)\vee(a\wedge c)$, as required.
\end{proof}

As in \cite[Definition 7.1]{Abbott1969}, when $B$ is a meet-semilattice and every principal ideal $b^\geq$ is a Boolean algebra, we call $B$ a \emph{semi-Boolean algebra}.\footnote{This differs from \cite[Definition 2.2]{Steinberg2010}, where semi-Boolean algebras are merely posets such that every principal ideal is a Boolean algebra.}  Every semi-Boolean algebra $B$ has a minimum $0$, for if $a\in B$ and $0$ is the minimum of $a^\geq$ then, for any $b\in B$, we have $b\wedge0\in 0^\geq\subseteq a^{\geq}$ so $0\leq b\wedge 0\leq b$.  Also, as $B$ is a meet-semilattice, any join in a principal ideal $b^\geq$ remains a join in the larger poset $B$. Indeed, let $s,t\in b^\geq$ and let $a$ be their join in $b^\geq$. To show that their join in $B$ exists and equals $a$, let $d\in B$ be such that $s,t\leq d$. Then $s,t\leq a\wedge d$ and since $a\wedge d \in b^\geq$ we have $a\leq a\wedge d$. Hence $a\leq d$, as required. Thus any semi-Boolean algebra $B$ is a \emph{conditional join-semilattice} in the sense that all bounded pairs have joins, i.e.
\[\tag{Conditional Join-Semilattice}a,b\leq c\qquad\Rightarrow\qquad a\vee b\text{ exists}.\]
Moreover, any semi-Boolean algebra $B$ is \emph{relatively pseudocomplemented} in that, for any $a,b\in B$, there is a (necessarily unique) \emph{relative pseudocomplement} $b\setminus a\in B$ satisfying
\[\tag{Relative Pseudocomplement}b\setminus a=\max(b^\geq\cap a^\top).\]
Indeed, $b\setminus a$ is precisely the complement of $a\wedge b$ in $b^\geq$.

Stone's classic representation of Boolean algebras as clopen subsets of Stone spaces (see \cite{Stone1937}) extends to semi-Boolean algebras, as we now proceed to show.  First recall that a \emph{filter} in a poset $B$ is a non-empty subset $F\subseteq B$ satisfying both
\[\tag{Up-set}\label{Upset}F^\leq\subseteq F\]
and
\[\tag{Directed}\label{Directed}a,b\in F\quad\Rightarrow\quad F\cap a^\geq\cap b^\geq\neq\emptyset.\]
If $B$ is a meet-semilattice then an up-set $F$ is directed and hence a filter precisely when it is closed under taking meets, i.e.~when $F\wedge F\subseteq F$.  We say $F$ is \emph{proper} when $F\neq B$.  If $B$ is a meet-semilattice with minimum $0$ then a filter $F\subseteq B$ is proper precisely when
\[
\{b\in F\mathrel{|}0\in b\wedge F\} \subseteq B\setminus F.
\]
An \emph{ultrafilter} is a maximal proper filter, which can also be characterised by the reverse inclusion above (e.g.~see also \cite[Lemma 3.2]{Exel2009}).

\begin{lemma}\label{UltrafilterComplements}
    A proper filter $U$ in a meet-semilattice $B$ with minimum $0$ is an ultrafilter if and only if
    \[B\setminus U\subseteq\{b\in B\mathrel{|}0\in b\wedge U\}.\]
\end{lemma}
\begin{proof}
    Take $a \in B \setminus U$.  If $0 \notin a \wedge U$ then $(a\wedge U)^\leq$ is a proper filter containing $U$.  If $U$ is an ultrafilter then maximality implies $a\in(a\wedge U)^\leq=U$, showing $B\setminus U\subseteq U$.

    Conversely, say $B\setminus U\subseteq \{b\in B\mathrel{|}0\in b\wedge U\}$.
    Suppose $U \subseteq V$ for some filter $V$.
    If $v \in V \setminus U \subseteq B \setminus U$, then there is some $u \in U$ such that $0 = v \wedge u$.
    Since $U \subseteq V$, we have $0 \in V$, and so $U$ is an ultrafilter.
\end{proof}

The \emph{ultrafilter space} of $B$ is the set of ultrafilters $\mathcal{U}(B)$ topologised by the basis $(O_b)_{b\in B}$ where
\[O_b=O_b^B=b^\in\cap\mathcal{U}(B)=\{U\in\mathcal{U}(B)\mathrel{|}b\in U\}.\]

Ultrafilter spaces always satisfy the following separation property.

\begin{proposition}\label{T1}
    The ultrafilter space $\mathcal{U}(B)$ of any poset $B$ is $T_1$.
\end{proposition}

\begin{proof}
    Given distinct $V,W\in\mathcal{U}(B)$, we have $v\in V\setminus W$ by maximality so $V\in O_v\not\ni W$.
\end{proof}

If $B$ has a minimum $0$, restricting to any principal ideal does not affect the topology.

\begin{proposition}\label{PrincipalRestriction}
    If $B$ is a poset with minimum $0$ then, for any $a\in B$, we have mutually inverse homeomorphisms from $O_a$ to $\mathcal{U}(a^\geq)$ and vice versa given by
    \[U\mapsto U\cap a^\geq\qquad\text{and}\qquad U\mapsto U^\leq.\]
\end{proposition}

\begin{proof}
    If $F\subseteq B$ is a filter containing $a$ then $F\cap a^\geq$ is a filter in $a^\geq$.  Moreover, for any $f\in F$, we must have $b\in F$ with $b\leq a,f$ so $f\in b^\leq\subseteq(F\cap a^\geq)^\leq$, showing that $(F\cap a^\geq)^\leq=F$.  Conversely, if $F$ is a filter in $a^\geq$ then $F^\leq$ is a filter in $B$ containing $a$ with $F^\leq\cap a^\geq=F$.  Thus the given operations are indeed mutually inverse bijections on filters.  They are also immediately seen to respect properness (as the proper filters are precisely those avoiding $0$), the containment relation $\subseteq$ and hence maximality so they must restrict to bijections of ultrafilters as well.  Lastly note that $O_a$ has a basis of sets of the form $(O_b)_{b\leq a}$.  But these subsets of $\mathcal{U}(B)$ correspond to the basis $(O^{a^\geq}_b)_{b\leq a}$ of $\mathcal{U}(a^\geq)$ under the given bijections.  Thus these maps are indeed homeomorphisms.
\end{proof}

Recall that a space is \emph{$0$-dimensional} if it has a basis of clopen sets.  Any $0$-dimensional $T_1$ space is automatically Hausdorff (and hence totally disconnected), which thus applies to $\mathcal{U}(B)$ in the following situation.

\begin{proposition}\label{0dim}
    If $B$ is a meet-semilattice with minimum $0$ then $\mathcal{U}(B)$ is $0$-dimensional.
\end{proposition}

\begin{proof}
    It suffices to show that each $O_a$ is closed.  Accordingly, take any $U\in\mathcal{U}(B)\setminus O_a$.  This means $a\notin U$ so we must have $u\in U$ with $a\mathrel{\top}u$ by \Cref{UltrafilterComplements}.
    Also, $a \mathrel{\top} u$  implies $O_a \cap O_u = \emptyset$, so $U\in O_u \subseteq \mathcal{U}(B)\setminus O_a$. Hence $\mathcal{U}(B)\setminus O_a=\bigcup\{O_b\mathrel{|}a\mathrel{\top}b\}$ is open, i.e.~$O_a$ is closed.
\end{proof}

\begin{theorem}\label{LocallyCompact}
    If $B$ is a semi-Boolean algebra then $\mathcal{U}(B)$ is locally compact.
\end{theorem}

\begin{proof}
    For each $b\in B$, Stone's original representation theorem says that $\mathcal{U}(b^\geq)$ is compact and hence so too is $O_b$, by \Cref{PrincipalRestriction}, i.e.~$\mathcal{U}(B)$ has a basis of compact open sets so $\mathcal{U}(B)$ is locally compact.
\end{proof}

\begin{remark} Combined with \Cref{T1,0dim}, this means $\mathcal{U}(B)$ locally compact, Hausdorff and $0$-dimensional when $B$ is a semi-Boolean algebra.  This can be viewed as a special case of a recent more general result for difference-restriction algebras of  \cite{BKL25}. Specifically, semi-Boolean algebras are precisely subtraction algebras $({\mathfrak{A}}, \cdot, -)$ where $\cdot$ is the meet and $-$ the relative complement operations (see \cite{BKL25} and references therein). As noted in \cite{BKL25}, setting $\triangleright = \cdot$ yields a difference--restriction algebra structure on a subtraction algebra ${\mathfrak{A}}$, hence subtraction algebras form a subclass of difference-restriction algebras. The Hausdorff \'etale space $\pi: X \twoheadrightarrow X_0$ of ultrafilters of a subtraction algebra ${\mathfrak{A}}$ is simply an identity map on a Hausdorff $0$-dimensional and locally compact space $X_0$, which is precisely the ultrafilter space of ${\mathfrak{A}}$, see \cite[Section 6]{BKL25}.
\end{remark}

The compact open subsets of any Hausdorff space form a semi-Boolean algebra ordered by inclusion, one which is even a join-semilattice.  Conversely, while a semi-Boolean algebra need only be a conditional join-semilattice, it can still always be faithfully represented as a basis of compact open subsets of some Hausdorff space, namely its ultrafilter space.

\begin{theorem}\label{StoneRep}
    If $B$ is a semi-Boolean algebra then the map $b\mapsto O_b$ is an order isomorphism from $B$ onto a basis of compact open subsets of $\mathcal{U}(B)$.
\end{theorem}

\begin{proof}
    For each $b\in B$, Stone's original representation theorem says that $\mathcal{U}(b^\geq)$ is compact and Hausdorff and hence so too is $O_b$, by \Cref{PrincipalRestriction}.  Moreover, these sets form a basis, by the definition of the topology of the ultrafilter space.  To see that $b\mapsto O_b$ is an order isomorphism note that, as in the proof of \Cref{0dim}, for all $a,b\in B$,
    \[O_b\setminus O_a=\bigcup\{O_c\mathrel{|}a\mathrel{\top}c\leq b\}=O_{b\setminus a}.\]
    Thus $O_b\subseteq O_a$ precisely when $O_{b\setminus a}=\emptyset$.  This is equivalent to $b\setminus a=0$ and hence to $b\leq a$, as the Kuratowski-Zorn lemma yields $U\in O_{b\setminus a}$ whenever $b\setminus a\neq0$.
    \end{proof}
     
\begin{remark} Theorem \ref{StoneRep} can be viewed as a special case of \cite{BKL25}, specifically of the construction and the properties of the functor $F$ \cite[Section 3.1]{BKL25}, when a difference restriction algebra therein is a subtraction algebra.
\end{remark}

For any meet-semilattice $B$ and any $a,b\in B$, we know that $O_{a\wedge b}=O_a\cap O_b$, as $a\wedge b\in U$ if and only if $a,b\in U$, for any (ultra)filter $U$ in $B$.  In other words, the map $b\mapsto O_b$ preserves meets.  If $B$ is semi-Boolean then $b\mapsto O_b$ also preserves joins whenever they exist.  
That is, $O_{a \vee b} = O_a \cup O_b$ for all $a, b \in B$ with a join $a \vee b$.
This follows from the fact that every ultrafilter $U$ in a semi-Boolean algebra $B$ is a \emph{prime filter} in the sense that $U$ is a proper filter and, for all $a,b\in B$ with a join $a\vee b$,
\[\tag{Prime}a\vee b\in U\qquad\Rightarrow\qquad a\in U\text{ or }b\in U.\]

\begin{proposition}\label{UltraPrime}
    In any semi-Boolean algebra $B$, ultrafilters and prime filters coincide.
\end{proposition}

\begin{proof}
    By \Cref{UltrafilterComplements}, the ultrafilters in $B$ are precisely those proper filters $U\subseteq B$ such that
    \[B\setminus U \subseteq \{b\in B\mathrel{|}0\in b\wedge U\}.\]
    Thus if $U$ is an ultrafilter and we have bounded $a,b\in B\setminus U$ then we have $v,w\in U$ with $a\wedge v=0=b\wedge w$.  Setting $x=v\wedge w\in U$, it follows that $a\wedge x=0=b\wedge x$.  
    Suppose $c \leq x, a \vee b$. Because $(a \vee b)^\geq$ is distributive, we have $(a \wedge c) \vee (b \wedge c) = (a \vee b) \wedge c = c$. Also, $a \wedge c \leq a \wedge x = 0$, so $a \wedge c = 0$. Similarly, $b \wedge c = 0$, and so $0 = 0 \vee 0 = c$. Thus, $(a \vee b) \wedge x = 0$. It follows $a\vee b\notin U$, showing that $U$ is prime.
    
    Conversely, if a proper filter $U\subseteq B$ is not an ultrafilter then there is a $b\in B\setminus U$ with $0\notin b\wedge U$. 
    Take any $u\in U$ note that $b\wedge u\notin U$, as $b\notin U$, and $0\notin b\wedge U\supseteq b\wedge u\wedge U$.  So replacing $b$ with $b\wedge u$ if necessary, we may assume $b\leq u$.  But then $b\vee(u\setminus b)=u\in U$, and $u \setminus b \not\in U$ -- otherwise $0 = b \wedge (u \setminus b) \in b \wedge U$. Thus $b, u \setminus b \not\in U$, and so $U$ is not prime.
\end{proof}

\section{Semigroupoids and Projections}\label{s:3}

As in \cite[\S1]{Bice2024}, by a \emph{semigroupoid} we mean a set $S$ on which we have an associative product defined on some pairs $S^2\subseteq S\times S$.  This means that $(ab)c$ is defined precisely when $a(bc)$ is defined, in which case $(ab)c=a(bc)$, which we then simply write as $abc$.
Note that this is more general than the semigroupoids in \cite{Exel2008}, which have the additional requirement that if $ab$ and $bc$ are defined then so are $(ab)c$ and $a(bc)$.

Let us assume now that we have a semigroupoid $S$ together with a distinguished commutative subsemigroupoid $P\subseteq S$ of idempotents in $S$ that we call \emph{projections}.

\subsection{The Canonical Order}

The projections $P$ define an order $\leq$ on $S$ where
\[a\leq b\qquad\Leftrightarrow\qquad a\in bP.\]
Strictly speaking this is not a partial order, as it is only reflexive on $SP$.  However, $\leq$ is transitive, as $PP\subseteq P$, and also antisymmetric -- if we have $a,b\in S$ and $p,q\in P$ with $a=bp$ and $b=aq$ then $a=bp=bpp=ap$ and so $a=bp=aqp=apq=aq=b$.

We immediately see $\leq$ respects products on the left, i.e.~if $a,b,c\in S$ and $ca$ is defined,
\begin{equation}\label{OrderPreservingLeft}
    a\leq b\qquad\Rightarrow\qquad ca\leq cb.
\end{equation}
On projections, it reduces to the usual ordering, since~for all $a\in S$ and $p\in P$,
\[p\leq a\qquad\Leftrightarrow\qquad p=ap.\]
Indeed, if we have $q\in P$ with $p=aq$ then $p=aqq=pq$ so $p=pp=aqp=apq=ap$.  Also any element bounded by a projection must again be a projection, i.e.
\[P^\geq=P.\]
Indeed, if $a\leq p\in P$ then $a\in pP\subseteq PP\subseteq P$.  Also, when products in $P$ are defined they are necessarily meets in $S$, i.e.~whenever $q,r\in P$ and $(q,r)\in S^2$,
\[qr=q\wedge r.\]
Indeed, certainly $qr\leq q,r$ while $p\leq q,r$ implies $p\in P$ so $p=pr=pqr$ and hence $p\leq qr$.

\subsection{Source and Range Projections}

Define relations $\mathsf{S},\mathsf{R}\subseteq S\times P$ by
\begin{align*}
    a\mathrel{\mathsf{S}}p\qquad&\Leftrightarrow\qquad ap=a.\\
    a\mathrel{\mathsf{R}}p\qquad&\Leftrightarrow\qquad pa=a.
\end{align*}
So $a^\mathsf{S}=\{p\in P\mathrel{|}ap=a\}$ and $a^\mathsf{R}=\{p\in P\mathrel{|}pa=a\}$.  When these sets have minima, we call 
the latter the \emph{source projection} and \emph{range projection} of $a$ respectively.  They are necessarily unique whenever they exist, in which case we may unambiguously denote them by
\[\mathsf{s}(a)=\min(a^\mathsf{S})\qquad\text{and}\qquad\mathsf{r}(a)=\min(a^\mathsf{R}).\]
They do at least exist for any $p\in P$, as then $p^\mathsf{S}=p^\mathsf{R}=p^\leq$ and hence $\mathsf{s}(p)=\mathsf{r}(p)=p$.  Source and range projections also respect the ordering, i.e.~if $\mathsf{s}(a)$ and $\mathsf{s}(b)$ exist then
\begin{equation}
\label{OrderPreservingSource}
    a\leq b\qquad\Rightarrow\qquad\mathsf{s}(a)\leq\mathsf{s}(b).
\end{equation}
Indeed, if $q\in P$, $a=bq$ and $p\in b^\mathsf{S}$ then $a=bq=bpq=bqp=ap$, i.e.~$p\in a^\mathsf{S}$.  Thus $b^\mathsf{S}\subseteq a^\mathsf{S}$ and hence $\mathsf{s}(a)\leq\mathsf{s}(b)$.  Likewise, if $\mathsf{r}(a)$ and $\mathsf{r}(b)$ exist then $a\leq b$ implies $\mathsf{r}(a)\leq\mathsf{r}(b)$.

\subsection{Inverses}
\label{Inverses}

We call $a,b\in S$ \emph{inverses} of each other if
\[\mathsf{r}(a)=ab=\mathsf{s}(b)\qquad\text{and}\qquad\mathsf{r}(b)=ba=\mathsf{s}(a).\]
In fact, it suffices for these conditions to hold just with source or range projections.  For example, if $ab=\mathsf{s}(b)$ and $ba=\mathsf{s}(a)$ then $a=a\mathsf{s}(a)=aba=\mathsf{s}(b)a$ and $pa=a$ implies $\mathsf{s}(b)=ab=pab=p\mathsf{s}(b)$, i.e.~$\mathsf{s}(b)\leq p$, showing that $\mathsf{s}(b)=\min(a^\mathsf{R})=\mathsf{r}(a)$ and, likewise, $\mathsf{s}(a)=\mathsf{r}(b)$.  Thus the partial inverses in restriction semigroups as defined in \cite[\S4.2]{Kudryavtseva2025} are precisely the inverses as defined here.  Similarly, if $\mathsf{r}(a)=ab$ and $\mathsf{r}(b)=ba$, then $\mathsf{r}(a)=\mathsf{s}(b)$ and $\mathsf{r}(b)=\mathsf{s}(a)$.
Also, it suffices to have $\mathsf{r}(a)=ab=\mathsf{s}(b)$ and $ba\in P$, for in that case $aba=\mathsf{r}(a)a=a$, and $ap=a$ implies $bap=ba$, i.e.~$ba\leq p$, showing that $ba=\min(a^\mathsf{S})=\mathsf{s}(a)$ and, likewise, $ba=\mathsf{r}(b)$.  Furthermore, inverses are unique whenever they exist, for if $b'$ is another inverse of $a$ then $b=b\mathsf{s}(b)=b\mathsf{r}(a)=bab'=\mathsf{s}(a)b'=\mathsf{r}(b')b'=b'$.  Thus when $a$ has an inverse we can unambiguously denote it by $a^{-1}$.  We denote the invertible elements by
\[S^\times=\{a\in S\mathrel{|}a^{-1}\text{ exists}\}.\]

\vspace{10pt}
We will now focus on semigroupoids where every element has a source projection.  In particular this implies $S=SP$ so the canonincal order $\leq$ is indeed reflexive and hence a partial order.  We will be particularly concerned with two special cases:
\begin{enumerate}
    \item Semigroups, where the product is defined everywhere and the projections are just any given commutative idempotent subsemigroup satisfying certain properties.
    \item Semicategories, where the product is not defined everywhere and where the projections are taken to be the units, often with a topology satisfying certain properties.
\end{enumerate}
We now discuss these two cases in more detail.

\subsection{Restriction Semigroups}
Here we briefly examine restriction semigroups -- for more information see the survey \cite{Gould2012} or \cite{Kudryavtseva2025} and the references therein.

To start with, say we have a semigroup $S$ with a distinguished commutative subsemigroup $P$ of idempotents.  Note $P$ is then a meet-semilattice such that $q\wedge r=qr\in PP\subseteq P$, for all $q,r\in P$.  If every element of $S$ has a source projection then joins of elements of $P$ must also lie in $P$ whenever they exist, i.e.~$P\vee P=P$.  Indeed, if $q,r\in P$ and $q\vee r$ exists then $q,r\leq q\vee r$ so $q=\mathsf{s}(q)\leq\mathsf{s}(q\vee r)$ and $r=\mathsf{s}(r)\leq\mathsf{s}(q\vee r)$ and hence $q\vee r\leq\mathsf{s}(q\vee r)$, which implies $q\vee r\in P^\geq=P$.

We call $S$ an \emph{Ehresmann semigroup} if each element of $S$ has a source projection and
\[\tag{Ehresmann Semigroup}\mathsf{s}(ab)=\mathsf{s}(\mathsf{s}(a)b),\]
for all $a,b\in S$.  
It is readily verified that this definition is equivalent to the definition of Ehresmann semigroups in \cite[Definition 3.1]{Kudryavtseva2025} since $S$ is an Ehresmann semigroup with a unary operation $^*$ in the sense of \cite{Kudryavtseva2025}, if and only if $S$ is Ehresmann relative to $P = S^*$ in our sense, where $\mathsf{s}(a)=a^*$.
It is known that the order relation in any Ehresmann semigroup $S$ can be characterised by
\[a\leq b\qquad\Leftrightarrow\qquad a=b\mathsf{s}(a).\]
Indeed, if we have $p\in P$ with $a=bp$ then $\mathsf{s}(a)=\mathsf{s}(bp)=\mathsf{s}(\mathsf{s}(b)p)=\mathsf{s}(b)p$ and hence $b\mathsf{s}(a)=b\mathsf{s}(b)p=bp=a$.  Consequently, for all $a,b\in S$,
\begin{equation}\label{TrivialRestriction}
    a\leq b\quad\text{and}\quad\mathsf{s}(b)\leq\mathsf{s}(a)\quad\Rightarrow\quad a = b.
\end{equation}
Also, bounded pairs have meets, i.e.~for all $a,b,c\in S$, 
\begin{equation}\label{ConditionalMeets}
    a,b\leq c\qquad\Rightarrow\qquad a\wedge b = b\mathsf{s}(a) = a\mathsf{s}(b).
\end{equation}
Indeed, because $a = c\mathsf{s}(a)$ and $b = c\mathsf{s}(b)$, we have 
$b\mathsf{s}(a) = c\mathsf{s}(b)\mathsf{s}(a) = c\mathsf{s}(a)\mathsf{s}(b)= a\mathsf{s}(b) \leq a$, 
and so $a\mathsf{s}(b) = b\mathsf{s}(a) \leq a, b$.  Moreover, if $d \leq a, b$, then $b\mathsf{s}(a)\mathsf{s}(d) = b\mathsf{s}(d)\mathsf{s}(a)=d\mathsf{s}(a) = d$, so $d \leq b\mathsf{s}(a)$.

We call $S$ a \emph{restriction semigroup} if $S$ is an Ehresmann semigroup satisfying
\[\tag{Restriction Semigroup}p\in P\quad\text{and}\quad b\in S\qquad\Rightarrow\qquad pb\leq b.\]
To emphasise the distinguished projections $P$, we may say $S$ is a restriction semigroup relative to $P$.
In particular, $\mathsf{s}(a)b\leq b$ so 
\[\mathsf{s}(a)b=b\mathsf{s}(\mathsf{s}(a)b)=b\mathsf{s}(ab),
\] for all $a,b\in S$.  
Conversely, if $S$ satisfies the identity $\mathsf{s}(a)b=b\mathsf{s}(ab)$, then for all $p \in P$ and $b \in S$, $pb=\mathsf{s}(p)b=b\mathsf{s}(pb)$, i.e. $pb\leq b$, and so our definition of restriction semigroups is equivalent to that of \cite{Kudryavtseva2025}.
In this case the ordering also respects products from the right, i.e.~for all $a,b,c\in S$,
\begin{equation}\label{OrderPreservingRight}
    a\leq b\qquad\Rightarrow\qquad ac\leq bc.
\end{equation}
Indeed, if $a=b\mathsf{s}(a)$ then $ac=b\mathsf{s}(a)c=bc\mathsf{s}(ac)\in bcP$.
The conditional meets from \eqref{ConditionalMeets} respect products from the right too, i.e.~for all $a,b,c,d\in S$,
\begin{equation}\label{MeetPreservingRight}
a,b\leq c\quad\Rightarrow\quad ad\wedge bd=(a\wedge b)d.
\end{equation}
Indeed, if $a,b\leq c$, then $ad, bd\leq cd$ by \eqref{OrderPreservingRight}, so we have $a\wedge b=b\mathsf{s}(a)$ and $ad\wedge bd=bd\mathsf{s}(ad)$ by \eqref{ConditionalMeets}.  Since $S$ is a restriction semigroup, we then have $ad\wedge bd=b\mathsf{s}(a)d=(a\wedge b)d$, as desired.

Now we turn our attention to invertible elements of restriction semigroups. We call a restriction semigroup $S$ an \emph{inverse semigroup} if $S = S^\times$, which is indeed an inverse semigroup in the usual sense, see \cite[Lemma 4.6]{Kudryavtseva2025}.
Furthermore, if $S$ is a restriction semigroup, then the invertibles $S^\times$ form an inverse subsemigroup. Furthermore, it is immediate that $P\subseteq S^\times$ with $e^{-1} = e$ for all $e\in P$. It follows that for all $b\in S^\times$ and $a\leq b$ we have  $a=b\mathsf{s}(a) \in S^\times$, showing that
\begin{equation}\label{IdealInvertibles}
S^{\times\geq}=S^\times.
\end{equation}
Also, since $S^\times$ is an inverse semigroup, we have $a^{-1} = (b\mathsf{s}(a))^{-1} = \mathsf{s}(a)^{-1}b^{-1} = \mathsf{s}(a)b^{-1}$.

\subsection{Semicategories}
For any semigroupoid $S$, we define the \emph{units} of $S$ by
\[S^0=\{u\in S\mathrel{|}\forall a\in S\ (((u,a)\in S^2\Rightarrow ua=a)\text{ and }((a,u)\in S^2\Rightarrow au=a))\}.\]
In other words, units are neutral with respect to products whenever they are defined.  Note that units are idempotent as long as the product is defined with at least some element of the semigroupoid, e.g. if $u\in S^0$ and $(a,u)\in S^2$ then $a=au=(au)u=a(uu)$ and hence $u=uu$. Moreover, for any other $t\in S^0$ with $(a,t)\in S^2$, we see that $a=at=a(ut)$ and hence $u=ut=t$. 
Denoting by $P$ the set of units $u$ such that there is $a\in S$ satisfying $(u,a)\in S^2$ or $(a,u)\in S^2$, we have shown that $P$ consists entirely of idempotents and $P = S^0S^0$. Furthermore, for every $a\in S$ there is at most one $u\in P$ such that $au=a$.
So we can take $P$ to be the set projections, in which case $a^\mathsf{S}$ and $a^\mathsf{R}$ contain at most one projection, for all $a\in S$, and hence $\mathsf{s}$ and $\mathsf{r}$ are defined precisely on $SS^0$ and $S^0S$ respectively.  Note that the ordering in this case is trivial in the sense that $a\leq b$ implies $a=b$.

Moreover, if $\mathsf{s}(a)$ and $ab$ are defined then $ab=(a\mathsf{s}(a))b=a(\mathsf{s}(a)b)$ so $\mathsf{r}(b)$ is also defined and is equal to $\mathsf{s}(a)$.  We call $S$ a \emph{semicategory} if each $a\in S$ has a source unit $\mathsf{s}(a)$ and the converse also holds, meaning that $ab$ is defined whenever $\mathsf{s}(a)=\mathsf{r}(b)$, i.e.
\[\tag{Semicategory}\label{Semicategory}S=SS^0\qquad\text{and}\qquad S^2=\{(a,b)\mathrel{|}\mathsf{s}(a)=\mathsf{r}(b)\}.\]
For any $(a, b) \in S^2$ in a semicategory $S$, because $(ab)\mathsf{s}(ab)$ is defined, so is $b\mathsf{s}(ab)$, which implies $\mathsf{s}(ab) = \mathsf{s}(b)$. If $ab$ also has a range unit $\mathsf{r}(ab)$, then $\mathsf{r}(ab) = \mathsf{r}(a)$.
If $a, b \in S$ are such that $ab, ba \in P$, then $a$ and $b$ are inverses. 
Indeed, the elements of $P$ are idempotents, so $(ab)(ab)$ is defined, which equals $a(b(ab))$ and hence $b(ab)$ is defined by associativity.
Since source units are unique, $ab = \mathsf{s}(b)$. Similarly, $ab = \mathsf{r}(a)$, showing that $a$ and $b$ are inverses. 
Furthermore, if $S$ is a semicategory then $P=S^0$, because for every $u\in S^0$ we have that $u\mathsf{s}(u) = u$, so that $u\in P$.

\begin{remark} One can show that any semicategory is an Exel semicategory and can be graphed in the sense of \cite{Cordeiro2023}.
\end{remark}

A \emph{category} is a semicategory where each $a\in S$ has a range unit too, i.e.~$S=S^0S$ as well.
A \emph{groupoid} is a (semi)category where each element is invertible, i.e.~$S^\times = S$.

\section{\'Etale Semicategories}\label{s:4}

Now we consider certain topologies on our semigroupoids.

\begin{definition}
    A semigroupoid $S$ is \emph{topological} if it carries a topology where the source, range and product maps are continuous on $SS^0$, $S^0S$ and $S^2$ respectively.
\end{definition}

First let us note some connections between closed subsets and the Hausdorff property.

\begin{proposition}\label{ClosedPairs}
    For any topology on a semigroupoid $S$,
    \[S^2\text{ is closed in }S\times S\qquad\Rightarrow\qquad S^0\text{ is Hausdorff}.\]
    If $S$ is a topological category then the converse also holds.
\end{proposition}

\begin{proof}
    If $S^0$ is not Hausdorff then we have a net $(x_\lambda)\subseteq S^0$ converging to distinct $y,z\in S^0$.  Then $(x_\lambda,x_\lambda)$ is a net in $S^2$ converging to $(y,z)\in(S\times S)\setminus S^2$, showing that $S^2$ is not closed in $S\times S$.  Conversely, say $S$ is a topological category and $S^0$ is Hausdorff.  Take a net $(a_\lambda,b_\lambda)\subseteq S^2$ converging to $(a,b)\in S\times S$.  Then $\mathsf{s}(a_\lambda)=\mathsf{r}(b_\lambda)$, for all $\lambda$, and hence $\mathsf{s}(a)=\lim\mathsf{s}(a_\lambda)=\lim\mathsf{r}(b_\lambda)=\mathsf{r}(b)$ so $(a,b)\in S^2$, showing that $S^2$ is closed in $S\times S$.
\end{proof}

\begin{proposition}\label{ClosedUnits}
    Assume $S$ is a topological semicategory.  Then
    \[S\text{ is Hausdorff}\qquad\Rightarrow\qquad S^0\text{ is Hausdorff and closed in }S.\]
    If $S$ is even a groupoid with continuous inverse map $a\mapsto a^{-1}$ then the converse also holds.
\end{proposition}

\begin{proof}
    If $S^0$ is not closed in $S$ then we have a net $(a_\lambda)\subseteq S^0$ converging to some $a\in S\setminus S^0$.  But then $a_\lambda=\mathsf{s}(a_\lambda)\rightarrow\mathsf{s}(a)\in S^0$, showing that limits are not unique and hence $S$ is not Hausdorff.  Conversely, if $S$ is even a groupoid with continuous inverse map and $S$ is not Hausdorff then we have a net $(a_\lambda)\subseteq S$ converging to distinct $b,c\in S$.  Then $a_\lambda a_\lambda^{-1}$ is a net in $S^0$ converging to $bc^{-1}\notin S^0$, showing that $S^0$ is not closed.
\end{proof}

To say more, we need to restrict our attention to \'etale semicategories, which generalise the \'etale categories of \cite{Kudryavtseva2025}.
Note that \cite{Kudryavtseva2025} assumes \'etale categories have locally compact Hausdorff unit spaces, which we do not assume.

\begin{definition}
    A semicategory $S$ is \emph{\'etale} if it is topological, $S^0S$ is open in $S$ and the source map $\mathsf{s}$ is a locally injective open map, i.e.~a local homeomorphism.
\end{definition}

A \emph{section} of a semicategory is a subset on which the source map $\mathsf{s}$ is injective.  Similarly to \cite[Lemma 6.7]{Kudryavtseva2025} it follows that open sections form a basis of the topology of any \'etale semicategory.

\begin{proposition}\label{OpenProduct}
    The product in any \'etale semicategory $S$ is an open map on $S^2$.
\end{proposition}

\begin{proof}
    As $S$ is a semicategory, $S^2 = \{(a, b)\mathrel{|}\mathsf{s}(a)=\mathsf{r}(b)\}$, so for any open sections $O,N\subseteq S$,
    \[ON=\{ab\mathrel{|}a\in O,b\in N\text{ and }\mathsf{s}(a)=\mathsf{r}(b)\}.\]
    As $\mathsf{s}$ is open, $\mathsf{s}[O]$ is open. Since $\mathsf{r}$ is continuous and $S^0S$ is open in $S$, $\mathsf{r}^{-1}[\mathsf{s}[O]]$ is open in $S$, and so $\mathsf{r}^{-1}[\mathsf{s}[O]]\cap N$ is open. Hence $Q\coloneqq\mathsf{s}[\mathsf{r}^{-1}[\mathsf{s}[O]]\cap N]$ is open.  Now note that $ON$ is the image of $Q$ under the continuous map $a\mapsto\mathsf{s}|_O^{-1}(\mathsf{r}(\mathsf{s}|_N^{-1}(a)))\mathsf{s}|_N^{-1}(a)$, which is also a right inverse to the source map and hence open, by \cite[Proposition 2.2]{Bice2023}.
\end{proof}

If $S$ is an \'etale semicategory then $S^0=\mathsf{s}[S]$ is automatically open.  The invertibles $S^\times$ of $S$ must then also form an open set on which the inverse map is automatically continuous, as we now show.  In particular, if $S$ is an \'etale semicategory in which every element is invertible then $S$ is indeed an \'etale groupoid in the usual sense (for example, see \cite{Ste19}).

\begin{proposition}
    If $S$ is an \'etale semicategory then $S^\times$ is open and $a\mapsto a^{-1}$ is continuous.
\end{proposition}

\begin{proof}
    Since $S^0$ is open, and the product is continuous on $S^2$, the preimage $U$ of $S^0$ under the product is open in $S^2$.
    For any $a\in S^\times$, $(a, a^{-1}), (a^{-1}, a) \in U$.
    The open sections form a basis of the topology for $S$, so there are open sections $O$, $N$, $N'$ and $O'$ such that $(a, a^{-1}) \in S^2 \cap (O \times N) \subseteq U$ and $(a^{-1}, a) \in S^2 \cap (N' \times O') \subseteq U$.
    Replacing with $O\cap O'$ and $N\cap N'$ if necessary, we may assume that $O=O'$ and $N=N'$.  
    Let $N'' \coloneqq \mathsf{r}^{-1}[\mathsf{s}[O]]\cap N$, and let $O'' \coloneqq \mathsf{r}^{-1}[\mathsf{s}[N'']]\cap O$.
    Because $\mathsf{s}$ is open, $\mathsf{r}$ is continuous on $S^0S$ and $S^0S$ is open in $S$, we have that $N''$ is open in $S$.
    Similarly, $O''$ is open in $S$.
    By definition of inverses, $\mathsf{r}(a) = \mathsf{s}(a^{-1})$ and $\mathsf{r}(a^{-1}) = \mathsf{s}(a)$.
    Then, $a \in O''$ because $a \in O$.
    We claim $O'' \subseteq S^\times$ so that $S^\times$ is open in $S$.
    Fix $b \in O''$.
    Recall that it suffices to find $c \in S$ such that $bc, cb \in S^0$ because $S$ is a semicategory.
    By definition of $O''$, $\mathsf{r}(b) = \mathsf{s}(c)$ for some $c \in N''$.
    Notice $(c, b) \in S^2 \cap (N \times O) \subseteq U$, so $cb \in S^0$.
    It remains to show $bc \in S^0$.
    Because $c \in N''$, $\mathsf{r}(c) = \mathsf{s}(d)$ for some $d \in O$.
    Like before, we have $dc \in S^0$.
    Notice $cb \in S^0$ implies $\mathsf{s}(b) = cb = \mathsf{r}(c) = \mathsf{s}(d)$.
    Since $O$ is a section, $b = d$, and so $bc \in S^0$, as required.
    Now notice that the inverse map on $O''$ is given by the continuous map $b^{-1}\mapsto\mathsf{s}|_N^{-1}(\mathsf{r}(b))$, so the inverse map is continuous on $S^\times$.
\end{proof}

If $S$ is a Hausdorff \'etale semicategory then $S^0$ is not just open but also closed, by \Cref{ClosedUnits}.  We can actually generalise this to any open section.

\begin{proposition}\label{ClopenSection}
    If $S$ is a Hausdorff \'etale semicategory then every open section $O\subseteq S$ is closed within $S\mathsf{s}[O]$.
\end{proposition}

\begin{proof}
    Take any $a\in S\mathsf{s}[O]\setminus O$.  We then have a unique $b\in O$ with $\mathsf{s}(a)=\mathsf{s}(b)$.  As $S$ is Hausdorff, we have open $M,N\subseteq S$ with $a\in M$, $b\in N$ and $M\cap N=\emptyset$.  Then $\mathsf{s}(b)\in\mathsf{s}[N\cap O]$ and hence $a\in M\mathsf{s}[N\cap O]$.  But also $M\mathsf{s}[N\cap O]\cap O=\emptyset$ because if we had $c\in M\mathsf{s}[N\cap O]\cap O$ then $\mathsf{s}(c)=\mathsf{s}(d)$, for some $d\in O\cap N$, and hence $c=d$, as $O$ is a section, even though $c\in M$ and $d\in N$, a contradiction.  Thus $M\mathsf{s}[N\cap O]$ is an open neighbourhood of $a$ disjoint from $O$.  As $a$ was an arbitrary element of $S\mathsf{s}[O]\setminus O$, this shows that $O$ is closed in $S\mathsf{s}[O]$.
\end{proof}

Note that we can consider the subsets of any semigroupoid $S$ as a semigroup under the product $TU=\{tu\mathrel{|}(t,u)\in S^2\}$, for $T,U\subseteq S$.  
When $S$ is a semicategory, we even get a restriction semigroup by restricting to sections and taking the projections to be precisely the subsets of units $S^0$.  This even applies to (compact) open sections under suitable conditions.

\begin{proposition}\label{SemiSlices}
    The open sections of any \'etale semicategory $S$ form a restriction semigroup relative to the open subsets of $S^0$.  If in addition $S^2$ is closed in $S\times S$, then the compact open sections of $S$ form a restriction semigroup relative to the compact open subsets of $S^0$.
\end{proposition}

\begin{proof}
    First we note that a product of sections $O,N\subseteq S$ of any discrete semicategory $S$ is again a section.  Indeed, if $a,a'\in O$ and $b,b'\in N$ then $\mathsf{s}(ab)=\mathsf{s}(a'b')$ means $\mathsf{s}(b)=\mathsf{s}(b')$ so $b=b'$, as $N$ is a section.  Thus $\mathsf{s}(a)=\mathsf{r}(b)=\mathsf{r}(b')=\mathsf{s}(a')$ so $a=a'$, as $O$ is a section, and hence $ab=a'b'$.  This shows that $ON$ is also a section, which in turn shows that the sections form a semigroup.  If $O, N \subseteq S^0$, then $ON = O \cap N \subseteq S^0$, so the subsets of $S^0$ form a commutative subsemigroup of idempotents, which we take to be projections. The order on the semilattice of projections is then just the inclusion $\subseteq$. It follows that $\mathsf{s}(O)$ is the smallest projection $E$ with $O=OE$. As
    \[(a,b)\in S^2\quad\Leftrightarrow\quad\mathsf{s}(a)=\mathsf{r}(b)\quad\Leftrightarrow\quad\mathsf{s}(\mathsf{s}(a))=\mathsf{r}(b)\quad\Leftrightarrow\quad(\mathsf{s}(a),b)\in S^2,\]
    in which case $\mathsf{s}(ab)=\mathsf{s}(b)=\mathsf{s}(\mathsf{s}(a)b)$, we see that
     $\mathsf{s}[ON]=\mathsf{s}[\mathsf{s}[O]N]$, showing that sections form an Ehresmann semigroup.  Moreover, $MN\subseteq N$ whenever $M\subseteq S^0$ so $MN = N\mathsf{s}(MN)$ and thus sections even form a restriction semigroup.
    
    Now assume $S$ is also \'etale, which embraces also just considered discrete case.  If $O\subseteq S$ is open then so is $\mathsf{s}[O]$, as $\mathsf{s}$ is an open map.  If $N\subseteq S$ is also open then, by \Cref{OpenProduct}, so is $ON$.  Thus the open sections form a subsemigroup that is closed under source projections which is thus a restriction semigroup.  Moreover, if $O\subseteq S$ is compact then so is $\mathsf{s}[O]$, as $\mathsf{s}$ is continuous.  If $S^2$ is also closed and $N\subseteq S$ is also compact then so is $O\times N$ and hence $S^2\cap O\times N$.  As the product map is continuous, $ON$ is then also compact.  Thus the compact open sections also form a subsemigroup which is yet again a restriction semigroup. \end{proof}
    
It is then natural to ask whether we can recover $S$ from the semigroup of compact open sections.  To have any hope of doing this, the compact open sections should distinguish elements of $S$.  This naturally leads to the notion of an ample semicategory.

\begin{definition}
    An \emph{ample semicategory} is an \'etale semicategory where $S^2$ is closed in $S\times S$ and the compact open subsets form a basis for the topology.
\end{definition}

Because open sections form a basis for the topology of any \'etale semicategory, the compact open sections form a basis for the topology of any ample semicategory.
Also, the condition that $S^2$ is closed in $S \times S$ holds automatically for any topological category $S$ with Hausdorff $S^0$ by \Cref{ClosedPairs}.
Hence, every ample category in the sense of \cite[Definition 7.1]{Kudryavtseva2025} is an ample semicategory.

Similarly as in the proof of \cite[Proposition 7.3(1)]{Kudryavtseva2025} one can show that an \'etale semicategory $S$ with Hausdorff unit space $S^0$ is ample if and only if $S^0$ is a locally compact Stone space.

\section{\'Etale Semicategory Algebras}\label{s:5}

In this section, we define algebras over unital rings from \'etale semicategories, generalising the algebras of ample categories in \cite[\S9]{Kudryavtseva2025b}, which in turn generalise Steinberg algebras of ample groupoids \cite{Steinberg2010}.  
The algebras of ample categories were independently developed in \cite[\S A.1.3]{Cam25}, therein called Orloff algebras, and shown to include, up to isomorphism, the path algebras of directed graphs. 

It is common in the literature for Steinberg algebras of ample groupoids to be called \'etale groupoid algebras, even when the construction uses the fact that the unit space is totally disconnected, which implies the groupoid is ample.  We do not make such an assumption on the unit space.
We generalise Steinberg algebras to any Hausdorff \'etale semicategory $S$ where $S^2$ is closed (\Cref{KSalg}).

We will only assume the scalars come from a ring $K$ with unit $1$.  In particular, $K$ may not be commutative.  As such, by a \emph{$K$-algebra} we mean a $K$-bimodule $A$ with a product $A \times A \to A$, denoted by $(a, b) \mapsto ab$, that is distributive with respect to addition and associative on $K\times A\times A$, $A\times K\times A$, $A\times A\times K$ and $A\times A\times A$.

If $S$ is a topological space, we call a function $f:S\rightarrow K$ \emph{locally constant} if $f$ is continuous with respect to the discrete topology on $K$.  This means $f^{-1}\{k\}$ is open and hence clopen, for all $k\in K$.  In particular, $\mathrm{supp}(f)=S\setminus f^{-1}\{0\}$ is then clopen.  
We call $f:S\rightarrow K$ \emph{compactly supported} if $\mathrm{supp}(f)$ is contained in some compact set $L$.  If $f$ is also locally constant, it follows that $\mathrm{supp}(f)=\mathrm{supp}(f)\cap L$ is itself compact.  We denote the $K$-valued compactly supported locally constant functions on $S$ by $KS$.  If $K$ is a (unital) ring then $KS$ is immediately seen to be a $K$-bimodule under pointwise sums and (left and right) scalar products.

When $S$ is a Hausdorff ample semicategory, elements of $KS$ will be linear combinations of characteristic functions of compact clopen sections (see \Cref{Semicategory->BooleanCartan}).  As in \cite[\S9]{Kudryavtseva2025b}, it then follows that $KS$ will be a $K$-algebra with respect to the \emph{convolution product}
\[
    (fg)(a) = \sum_{a = bc}f(b)g(c),
\]
for each $f, g \in KS$ and $a \in S$.  We have to work a little harder to prove this when $S$ is merely \'etale, as we now set about showing.  First we show $\mathsf{s}$ is a covering map on compact open subsets.

\begin{lemma}\label{ParallelSections}
    If $S$ is a Hausdorff \'etale semicategory, $x\in S^0$ and $C\subseteq S$ is compact and open then we have open $O\subseteq S^0$ containing $x$ and disjoint open sections $O_1,\ldots,O_n$ such that $C\cap SO=\bigcup_{k=1}^nO_k$ and $\mathsf{s}[O_k]=O$, for all $k$.
\end{lemma}

\begin{proof}
    As $C$ is compact and the open sections form a basis, we can cover $C$ by finitely many (open) sections.  Thus $Cx = C \cap \mathsf{s}^{-1}\{x\}$ is finite (possibly empty), say $Cx=\{a_1,\ldots,a_n\}$.  As $S$ is Hausdorff and \'etale, we have disjoint open sections $O_1,\ldots,O_n$ with $a_k\in O_k$, for all $k$.  Because $C$ is open, replacing $O_{k}$ with $O_{k} \cap C$ if necessary, we can assume $O_{k} \subseteq C$ for all $k$.  Then we have compact $D=C\setminus\bigcup_{k=1}^nO_k$ (so $D=C$ when $Cx=\emptyset$) with $x\notin\mathsf{s}[D]$.  Because $\mathsf{s}$ is continuous, $\mathsf{s}[D]$ is compact.  As $S$ is Hausdorff, $\mathsf{s}[D]$ is closed.  Also, $\mathsf{s}$ is open, so $\mathsf{s}[O_k]$ is open for all $k$.  We then have open $O=\bigcap_{k=1}^n\mathsf{s}[O_k]\setminus\mathsf{s}[D]$ containing $x$ such that $C\cap SO=\bigcup_{k=1}^nO'_k$ where $O'_k=O_kO$ and hence $\mathsf{s}[O'_k]=\mathsf{s}[O_k]O=O$, for all $k$.
\end{proof}

In \Cref{CharacteristicProduct} below, we consider the convolution $f\chi_O$ of $f \in KS$ and the characteristic function $\chi_O$ of an open section $O \subseteq S$.  Note that $\chi_O$ may not be in $KS$, but we will see that $f\chi_O$ is still a well-defined function on $S$.

\begin{lemma}\label{CharacteristicProduct}
    If $S$ is a Hausdorff \'etale semicategory, $S^2$ is closed, $f\in KS$ and $O\subseteq S$ is an open section with characteristic function $\chi_O$, then $f\chi_O$ is locally constant on $S\mathsf{s}[O]$.
\end{lemma}

\begin{proof}
    Take any $a\in O$.  If $\mathrm{supp}(f)a=\emptyset$ then we claim there is some open $N\subseteq O$ containing $a$ with $\mathrm{supp}(f)N=\emptyset$ or, equivalently, $\mathsf{s}[\mathrm{supp}(f)]N=\emptyset$.  If not then we have $(a_\lambda)\subseteq\mathsf{s}[\mathrm{supp}(f)]O$ with $a_\lambda\rightarrow a$.  Then $\mathsf{r}(a_\lambda)$ is defined for all $\lambda$ and has a subnet converging to some $x$ in the compact set $\mathsf{s}[\mathrm{supp}(f)]$.  But then $xa$ is defined, as $S^2$ is closed in $S\times S$, a contradiction.  This proves the claim, from which it follows that $f\chi_O[S\mathsf{s}[N]]=\{0\}$ and, in particular, $f\chi_O$ is continuous on the open set $S\mathsf{s}[N]$ containing $S\mathsf{s}(a)$.
        
    Now assume instead that $\mathrm{supp}(f)a\neq\emptyset$ and hence $\mathsf{r}(a)$ is defined and lies in the open set $\mathsf{s}[\mathrm{supp}(f)]$.  By \Cref{ParallelSections}, we have open $N\subseteq S^0$ containing $\mathsf{r}(a)$ and disjoint open sections $O_1,\ldots,O_n$ such that $\mathrm{supp}(f)\cap SN=\bigcup_{j=1}^nO_j$ and $\mathsf{s}[O_j]=N$, for all $j$.  As $f$ is locally constant, we can assume $f$ has some constant value $k_j\in K$ on each $O_j$ by replacing $N$ and each $O_j$ as follows.  For all $j$, $\mathsf{r}(a) \in N = \mathsf{s}[O_j]$, so there is some $d_j \in O_j$ such that $d_ja$ is defined, and we let $k_j = f(d_j)$.  Using that the open sections form a basis, for all $j$, there is an open section $O_j'$ such that $d_j \in O_j' \subseteq f^{-1}\{k_j\} \cap O_j$.  Then, it suffices to replace $N$ with $\bigcap_{j=1}^n \mathsf{s}[O_j']$ and $O_j$ with $O_j'N'$ for all $j$.  
    
Now consider the open subset $M=\mathsf{r}^{-1}[N]\cap O$ containing $a$ and note that (since $M\subseteq O$) that $f\chi_M = \sum_{j=1}^n k_j\chi_{O_jM}$.   As each $O_jM$ is an open section with $\mathsf{s}[O_jM]=\mathsf{s}[M]$, each $\chi_{O_jM}$ is locally constant on $S\mathsf{s}[M]$.  This means $f\chi_M$ and hence $f\chi_O$ is locally constant on the open set $S\mathsf{s}[M]$ containing $S\mathsf{s}(a)$.  As $a$ was arbitrary, $f\chi_O$ is thus locally constant on $S\mathsf{s}[O]$.
\end{proof}

\begin{theorem}\label{KSalg}
    If $S$ is a Hausdorff \'etale semicategory, $S^2$ is closed and $f,g\in KS$ then $fg\in KS$ too.  Thus $KS$ forms a $K$-algebra under convolution.
\end{theorem}

\begin{proof}
    First note that $fg$ is zero and, in particular, locally constant on the clopen set $S\setminus(S\mathsf{s}[\mathrm{supp}(g)])$.  To see that $fg$ is also locally constant on the clopen set $S\mathsf{s}[\mathrm{supp}(g)]$, take any $x\in\mathsf{s}[\mathrm{supp}(g)]$.  By \Cref{ParallelSections}, we have open $O\subseteq S^0$ containing $x$ and disjoint open sections $O_1,\ldots,O_n$ such that $\mathrm{supp}(g)\cap SO=\bigcup_{j=1}^nO_j$ and $\mathsf{s}[O_j]=O$, for all $j$.  Shrinking $O$ if necessary, like in the proof of \Cref{CharacteristicProduct}, we may further assume that $g$ is constant on $O_j$, for all $j$.  Thus on $SO$, $g$ is just $\sum_{j=1}^nk_j\chi_{O_j}$, for some $k_1,\ldots,k_n\in K\setminus\{0\}$.  By \Cref{CharacteristicProduct}, $f\chi_{O_j}$ is a well-defined function that is locally constant on $S\mathsf{s}[O_j]$ for all $j$.  Consider the sum of right scalar products $\sum_{j=1}^n f\chi_{O_j}k_j$, i.e.~$\sum_{j=1}^n (f\chi_{O_j})(a)k_j$ for all $a \in S$.  Fix $a \in SO$.  For all $j$, $f\chi_{O_j}$ is locally constant on $S\mathsf{s}[O_j] = SO$, which contains $a$, so there is an open $N_j \subseteq SO$ such that $a \in N_j$ and $f\chi_{O_j}$ is constant on $N_j$.  Then, $\sum_{j=1}^n f\chi_{O_j}k_j$ is constant on $\bigcap_{j=1}^n N_j \ni a$, and hence locally constant on $SO$.  For any $a \in SO$, if $a = bc$, then $c \in SO$ too, and so 
    \[
        (fg)(a) 
        = \sum_{a = bc} f(b)g(c) 
        = \sum_{a = bc} \sum_{j=1}^n f(b)k_j\chi_{O_j}(c)
        = \sum_{a = bc} \sum_{j=1}^n f(b)\chi_{O_j}(c)k_j
         = \sum_{j=1}^n (f\chi_{O_j})(a)k_j.
    \]
    Thus, $fg$ is just $\sum_{j=1}^n f\chi_{O_j}k_j$, and hence is locally constant, on $SO$.  As $x$ was arbitrary, $fg$ is thus locally constant on $S\mathsf{s}[\mathrm{supp}(g)]$.  Since $fg$ is locally constant on both the clopen set $S\mathsf{s}[\mathrm{supp}(g)]$ and its complement in $S$, $fg$ is locally constant.  Notice $\mathrm{supp}(fg) \subseteq \mathrm{supp}(f)\mathrm{supp}(g)$, which is compact because $S^2$ is closed, so $fg$ is compactly supported.  Thus, $fg \in KS$.  Because $(ab)c$ is defined if and only if $a(bc)$ is defined, the associativity of the multiplication in $K$ makes the convolution on $KS$ associative.  Similarly, distributivity in $K$ gives rise to distributivity in $KS$.
\end{proof}

\section{Topological representation}\label{s:6}

Recently in \cite{Kudryavtseva2025}, the classic Stone duality between generalized Boolean algebras and locally compact Hausdorff spaces was extended to Boolean restriction semigroups with local units and ample categories. The paper \cite{Kudryavtseva2025} provided a construction of an ample category from a preBoolean restriction semigroup (see also \cite{Kudryavtseva2025b} where an ample category is constructed from an arbitrary restriction semigroup with local units). Here we provide a construction similar to that in \cite{Kudryavtseva2025} for semi-Boolean restriction semigroups with binary meets, potentially with unbounded pairs of projections and no local units,  taking the approach via ultrafilters rather than that via germs used in \cite{Kudryavtseva2025,Kudryavtseva2025b}.\footnote{For a detailed explanation of equivalence of the ultrafilter and the germ approaches to constructing ample groupoids of inverse semigroups, see \cite{ACaHJL22}. We believe these approaches are equivalent for constructing semicategories from restriction semigroup, although this question was not treated in the literature.} 

For the rest of this section let us assume that
\begin{center}
    \textbf{$B$ is a restriction semigroup with projections $P$ containing a left zero $0$ for $B$.}
\end{center}

It is then easy to show (see \cite[Lemma 5.1]{Kudryavtseva2025}) that $0$ is also a right zero and the minimum element of $B$.

\begin{proposition}\label{FilterSource}
    For any $a\in B$, we have mutually inverse $\subseteq$-preserving maps from filters of $B$ containing $a$ to filters of $B$ containing $\mathsf{s}(a)$ and versa given by
    \[F\mapsto\mathsf{s}[F]^\leq\qquad\text{and}\qquad F\mapsto (aF)^\leq.\]
\end{proposition}

\begin{proof}
    For each filter $F$ containing $a$, $\mathsf{s}[F]^\leq$ is a filter containing $\mathsf{s}(a)$ because the source map is order-preserving, as per \eqref{OrderPreservingSource}. For each filter $F$ containing $\mathsf{s}(a)$, $(aF)^\leq$ is a filter containing $a$ because left multiplication is order-preserving, as per \eqref{OrderPreservingLeft}.  The maps are also immediately seen to respect inclusions.  Now take a filter $F$ containing $a$.  For any $f\in F$, we have $g\in F$ with $g\leq a,f$ and hence $f\geq g=a\mathsf{s}(g)\in a\mathsf{s}[F]$, showing that $F\subseteq(a\mathsf{s}[F])^\leq$.  But also $g=g\mathsf{s}(g)\leq a\mathsf{s}(f)$, showing that $a\mathsf{s}[F]\subseteq F^\leq\subseteq F$ and hence $(a\mathsf{s}[F])^\leq\subseteq F^\leq\subseteq F$.
    
    On the other hand, say $F$ is a filter containing $\mathsf{s}(a)$.  For any $f\in F$, we have $p\in F$ with $p\leq\mathsf{s}(a),f$.  As $\mathsf{s}(a)$ is a projection, so is $p$ and hence $p=pp\leq\mathsf{s}(a)f$, which implies $p=\mathsf{s}(p)\leq\mathsf{s}(\mathsf{s}(a)f)=\mathsf{s}(af)$.  This shows that $\mathsf{s}[aF]\subseteq F^\leq\subseteq F$ and hence $\mathsf{s}[aF]^\leq\subseteq F^\leq\subseteq F$.  But also $\mathsf{s}(ap)=\mathsf{s}(\mathsf{s}(a)p)=\mathsf{s}(a)p\leq p < f$, showing that $F\subseteq\mathsf{s}[aF]^\leq$.  Thus the given maps are indeed mutual inverses of each other.
\end{proof}

An important immediate consequence is the following.

\begin{corollary}\label{UltraSource}
    A filter $F\subseteq B$ is an ultrafilter if and only if $\mathsf{s}[F]^\leq$ is.
\end{corollary}

\begin{proof}
    Firstly, note that $F$ is proper if and only if $\mathsf{s}[F]^\leq$ is.  Take $a\in F$, and consider the associated maps from \Cref{FilterSource}.  Suppose $F$ is an ultrafilter and $\mathsf{s}[F]^\leq \subseteq G'$ for some proper filter $G'$.  Then, $G'$ is a filter containing $\mathsf{s}(a)$, so $G' = \mathsf{s}[G]^\leq$ for a filter $G$ containing $a$.  Because the maps are mutually inverse and order-preserving, $F \subseteq G$.  Since $\mathsf{s}[G]^\leq$ is proper, so is $G$, and so $F = G$, which implies $\mathsf{s}[F]^\leq = \mathsf{s}[G]^\leq$.  Hence, $\mathsf{s}[F]^\leq$ is an ultrafilter.  The converse holds by a similar argument.
\end{proof}

With this in hand we can show that the ultrafilters of $B$ form a semicategory.

\begin{proposition}\label{UltraUnit}
    The ultrafilters $\mathcal{U}(B)$ form a semicategory under the product
    \[V\cdot W=(VW)^\leq\text{ defined precisely when }0\notin VW.\]
    Moreover, the unit ultrafilters are precisely those containing a projection, i.e.
    \[\tag{Units}\label{Units}\mathcal{U}(B)^0=\bigcup_{p\in P}O_p=\{U\in\mathcal{U}(B)\mathrel{|}U\cap P\neq\emptyset\}.\]
\end{proposition}

\begin{proof}
    First we claim $(VW)^\leq$ is indeed an ultrafilter when $V$ and $W$ are and $0\notin VW$.  Because both left and right multiplication in $B$ are order-preserving -- see \eqref{OrderPreservingLeft} and \eqref{OrderPreservingRight}, $(VW)^\leq$ is a filter.  Using the defining property of restriction semigroups and that the source map is order-preserving, we can show $\mathsf{s}[W]\subseteq\mathsf{s}[VW]^\leq$.  Then, $0\notin VW$ implies $0\notin\mathsf{s}[(VW)^\leq]^\leq=\mathsf{s}[VW]^\leq\supseteq\mathsf{s}[W]^\leq$.  By \Cref{UltraSource}, $\mathsf{s}[W]^\leq$ is an ultrafilter so this inclusion must be an equality, i.e.~$\mathsf{s}[(VW)^\leq]^\leq$ and hence $(VW)^\leq$ is an ultrafilter, again by \Cref{UltraSource}.  For any other $X\in\mathcal{U}(B)$, note that if $(V\cdot W)\cdot X$ is defined then $0\notin(VW)^\leq X\supseteq VWX$, while if $0\notin VWX$ then certainly $0\notin VW$ and $0\notin(VWX)^\leq\supseteq(VW)^\leq X$, i.e.~$(V\cdot W)\cdot X$ is defined.  So $(V\cdot W)\cdot X$ is defined precisely when $0\notin VWX$, in which case $(V\cdot W)\cdot X=((VW)^\leq X)^\leq=(VWX)^\leq$.  But the same argument shows that $V\cdot(W\cdot X)$ is defined precisely when $0\notin VWX$, in which case $V\cdot(W\cdot X)=(VWX)^\leq=(V\cdot W)\cdot X$.  Thus $\mathcal{U}(B)$ is a semigroupoid.

    Next we claim that $V\in\mathcal{U}(B)$ is a unit precisely when $V\cap P\neq\emptyset$.  Indeed, if $p\in V\cap P$ and $0\notin VW$ then $V\cdot W=(VW)^\leq\supseteq(pW)^\leq\supseteq W$.  As $W$ is an ultrafilter, this inclusion must be an equality and hence $V\cdot W=W$.  The same argument shows $W\cdot V=W$ when $0\notin WV$, proving the claim.  For any $U\in\mathcal{U}(B)$, this means $\mathsf{s}[U]^\leq\in\mathcal{U}(B)^0$, again by \Cref{UltraSource}.  By \Cref{FilterSource}, $U=(u\mathsf{s}[U])^\leq=(u\mathsf{s}[U]^\leq)^\leq$ for any $u\in U$.  Notice $0\notin U\mathsf{s}[U]^\leq$ -- otherwise $0 = a\mathsf{s}(b)$ for some $a, b \in U$ and so there is a $c \in U \cap a^\geq \cap b^\geq$ satisfying $0 = a\mathsf{s}(b) \geq c\mathsf{s}(c) = c$, a contradiction.  Hence, $U\cdot\mathsf{s}[U]^\leq$ is defined, which means that $U$ has source unit $\mathsf{s}(U)=\mathsf{s}[U]^\leq$.  In particular, if $U\in\mathcal{U}(B)^0$ then $U=\mathsf{s}(U)\supseteq\mathsf{s}[U]$ and hence $U\cap P\neq\emptyset$, proving \eqref{Units}.

    Now say $V,W\in\mathcal{U}(B)$ and $W$ has a range unit $\mathsf{r}(W)=\mathsf{s}(V)$.  For all $v\in V$ and $w\in W$, it follows that $\mathsf{s}(v)w\in\mathsf{s}(V)\cdot W=\mathsf{r}(W)\cdot W=W\not\ni0$ so $\mathsf{s}(vw)=\mathsf{s}(\mathsf{s}(v)w)\neq0$ and hence $vw\neq0$.  Thus $0\notin VW$ so $V\cdot W$ is defined, showing $\mathcal{U}(S)$ is indeed a semicategory.
\end{proof}

We can also characterise the invertible ultrafilters much like the unit ultrafilters.

\begin{proposition}\label{UltraInvertible}
    The invertible ultrafilters are those containing an invertible of $B$, i.e.
    \[\tag{Invertibles}\label{Invertibles}\mathcal{U}(B)^\times=\bigcup_{a\in B^\times}O_a=\{U\in\mathcal{U}(B)\mathrel{|}U\cap B^\times\neq\emptyset\}.\]
\end{proposition}

\begin{proof}
    For any $a\in B^\times$, using that $(B^\times)^\geq = B^\times$ and that the inverse map on $B^\times$ is order-preserving, we can argue as in the proof of \Cref{FilterSource} to show that $F\mapsto(a^{-1}Fa^{-1})^\leq$ and $F\mapsto(aFa)^\leq$ yield mutually inverse $\subseteq$-preserving maps from filters containing $a$ to filters containing $a^{-1}$ and vice versa that identify proper filters with proper filters.  In particular, for any $U\in O_a$, we have $V \coloneqq (a^{-1}Ua^{-1})^\leq\in O_{a^{-1}}$.  We proceed to show $V$ is an inverse of $U$.  It suffices to show $U\cdot V,V\cdot U\in\mathcal{U}(B)^0$ because $\mathcal{U}(B)^0$ is a semicategory.  We first claim that $0\notin a^{-1}Ua^{-1}Ua^{-1}$.  Indeed, if we had $u,v\in U$ with $a^{-1}ua^{-1}va^{-1}=0$ then taking any $t\in U$ with $t\leq a,u,v$ would yield
    \[0=aa^{-1}ta^{-1}ta^{-1}a=aa^{-1}a\mathsf{s}(t)a^{-1}a\mathsf{s}(t)a^{-1}a=a\mathsf{s}(a)\mathsf{s}(t)\mathsf{s}(a)\mathsf{s}(t)\mathsf{s}(a)=a\mathsf{s}(t)=t\in U,\]
    contradicting the fact $U$ is a proper filter.  This proves the claim, from which it follows that $0\notin UV\cup VU$.  As $\mathsf{r}(a)=aa^{-1}\in UV$ and $\mathsf{s}(a)=a^{-1}a\in VU$, it follows that $U\cdot V,V\cdot U\in\mathcal{U}(B)^0$ by \Cref{UltraUnit}.  Thus $V$ is an inverse of $U$, showing that $O_a\subseteq\mathcal{U}(B)^\times$.

    Conversely, take any $U\in\mathcal{U}(B)^\times$ so we have $U^{-1}\in\mathcal{U}(B)$ with $U\cdot U^{-1},U^{-1}\cdot U\in\mathcal{U}(B)^0$.  By \eqref{Units}, we must have $a,b\in U$ and $a',b'\in U^{-1}$ with $aa',b'b\in P$.  Taking any $c\in U$ with $c\leq a,b$ and $c'\in U^{-1}$ with $c'\leq a',b'$, we then have $cc',c'c\in P$.  Let $d=cc'c\in UU^{-1}U\subseteq U$ and $d'=c'cc'\in U^{-1}UU^{-1}\subseteq U^{-1}$.  Then $dc'c=cc'cc'c=cc'c=d$ and $dp=d$ implies $c'cp=c'cc'cp=c'dp=c'd=c'cc'c=c'c$, i.e.~$c'c\leq p$, showing that $\mathsf{s}(d)=c'c=c'cc'cc'c=d'd$.  Likewise, $\mathsf{s}(d')=dd'$ so $d\in U\cap B^\times$, as required.
\end{proof}

When $B$ is also a semi-Boolean algebra, products from the left automatically distribute over joins when they exist, i.e.~$a(b\vee c)=ab\vee ac$ when $b$ and $c$ are bounded, see
 \cite[Lemma 5.6]{Kudryavtseva2025}. 
 
If $B$ is a semi-Boolean algebra and products from the right also distribute over existing joins, i.e.
\[(a\vee b)c=ac\vee bc,\]
whenever $a$ and $b$ are bounded, then we call $B$ a \emph{semi-Boolean restriction semigroup}.
Recall that $\mathsf{S}$ and $\mathsf{R}$ are the source and range relations on $B \times P$, so $U^\mathsf{S} = \{p\in P\mathrel{|}ap=a \text{ for some } a \in U\}$ and $U^\mathsf{R} = \{p\in P\mathrel{|}pa=a \text{ for some } a \in U\}$ for $U \subseteq B$.

\begin{theorem}\label{AmpleSemi}
    If $B$ is a semi-Boolean restriction semigroup then the ultrafilters $\mathcal{U}(B)$ form a Hausdorff ample semicategory.
    The source unit of any $U\in\mathcal{U}(B)$ is given by $\mathsf{s}(U)=U^{\mathsf{S}\leq}$, while $U$ has a range unit precisely when $U^\mathsf{R}\neq\emptyset$, in which case $\mathsf{r}(U)=U^{\mathsf{R}\leq}$.
\end{theorem}

\begin{proof}
    We saw in \Cref{UltraUnit} that $\mathcal{U}(B)$ is a semicategory such that the source unit of every $U\in\mathcal{U}(B)$ is $\mathsf{s}[U]^\leq$, which is the same as $U^{\mathsf{S}\leq}$, by the definition of source projections.  By \Cref{FilterSource} and \Cref{UltraSource}, for each $a\in B$, the source map $\mathsf{s}$ is then a bijection from $O_a$ onto $O_{\mathsf{s}(a)}$, showing that $\mathsf{s}$ is indeed a locally injective continuous open map.  To see the product is continuous, we show the preimage of $O_{c}$ is open for all $c \in B$.  Indeed, if $(U, V) \in \mathcal{U}(B)^2$ and $U \cdot V \in O_{c}$, then there are $a \in U$ and $b \in V$ such that $ab \leq c$, and so $(U, V) \in \mathcal{U}(B)^2 \cap (O_{a} \times O_{b})$, which is contained in the preimage of $O_{c}$ because $O_aO_b\subseteq O_{ab}\subseteq O_{c}$.

    Now we show that $U \in \mathcal{U}(B)$ has a range unit precisely when $U^\mathsf{R} \neq \emptyset$, in which case $\mathsf{r}(U) = U^{\mathsf{R}\leq}$.  Firstly, notice $0\notin U^{\mathsf{R}\leq}U$ in general.  Otherwise, we can take $p \in U^\mathsf{R}$ and $w \in U$ such that $pw = 0$ and $v \in U$ such that $pv = v$.  Then, there is some $x \in U$ such the $x \leq v, w$, and we have $x=v\mathsf{s}(x)=pv\mathsf{s}(x)=px\leq pw = 0$, contradicting the fact that $U$ is proper.  This also shows $0\notin U^{\mathsf{R}\leq}$.  Now suppose $U^\mathsf{R} \neq \emptyset$.  We show $U^{\mathsf{R}\leq}$ is an ultrafilter.  To see that it is a filter, suppose we have $p, q\in U^\mathsf{R}$ with $pv = v$ and $qw = w$ for some $v, w \in U$.  Taking $x \in U$ such that $x \leq v, w$, we have $pqx = pqw\mathsf{s}(x) = pw\mathsf{s}(x) = px = pv\mathsf{s}(x) = v\mathsf{s}(x) = x$, and hence $p\wedge q=pq\in U^\mathsf{R}$.  To see that it is even prime and hence an ultrafilter, by \Cref{UltraPrime}, take any bounded $a,b\in B$ with $a\vee b\in U^{\mathsf{R}\leq}$.  So we have $p\in U^\mathsf{R}$ with $p\leq a\vee b$, which means we have $u\in U$ with $u=pu=((p\wedge a)\vee(p\wedge b))u=(p\wedge a)u\vee(p\wedge b)u$.  As $U$ is prime, again by \Cref{UltraPrime}, it follows that $(p\wedge a)u\in U$ or $(p\wedge b)u\in U$ so $p\wedge a\in U^\mathsf{R}$ or $p\wedge b\in U^\mathsf{R}$ and hence $a\in U^{\mathsf{R}\leq}$ or $b\in U^{\mathsf{R}\leq}$, as required.  As $p\in U^\mathsf{R}\subseteq P$, it follows from \Cref{UltraUnit} that $U^{\mathsf{R}\leq}$ is a unit ultrafilter and hence the range unit of $U$, as $0\notin U^{\mathsf{R}\leq}U$.  Conversely, if $U$ has a range unit $W$, then there is some $p \in W \cap P$ by \Cref{UltraUnit}, and $pU \subseteq W \cdot U = U$.  Hence, taking any $a \in U$, we have $ppa = pa \in U$, so $p \in U^\mathsf{R} \neq \emptyset$.
    
    Now, whenever $U$ has a range unit, we can take $p \in U^\mathsf{R}$ with $pv = v$ for some $v \in U$, and it follows that $p\in V^\mathsf{R}$, for all $V\in O_v$, so $O_v$ is a neighbourhood of $U$ on which the range map $\mathsf{r}$ is always defined with image lying in $O_p$.  This shows that the range map is continuous and defined on an open subset of $\mathcal{U}(B)$.

    We have thus verified that $\mathcal{U}(B)$ is an \'etale semicategory.  As $B$ is a semi-Boolean algebra, $\mathcal{U}(B)$ is Hausdorff, by \Cref{T1,0dim}, and has a basis of compact open subsets, as noted in \Cref{StoneRep}.  It only remains to show that $\mathcal{U}(B)^2$ is closed in $\mathcal{U}(B)\times\mathcal{U}(B)$.  Accordingly, take nets $(V_\lambda),(W_\lambda)\subseteq\mathcal{U}(B)$ converging to $V$ and $W$ respectively with $0\notin V_\lambda W_\lambda$, for all $\lambda$.  If $v\in V$ and $w\in W$, i.e.~$V\in O_v$ and $W\in O_w$, then $V_\lambda\in O_v$ and $W_\lambda\in O_w$, for all sufficiently large $\lambda$.  Picking any such $\lambda$, this means that $v\in V_\lambda$ and $w\in W_\lambda$ so $vw\in V_\lambda W_\lambda\not\ni0$ and hence $vw\neq0$.  This shows that $0\notin VW$ so $(V,W)\in\mathcal{U}(B)^2$, as required.  Thus $\mathcal{U}(B)$ is indeed ample.
\end{proof}

We can thus represent $B$ on a Hausdorff ample semicategory as follows.
Recall from \Cref{SemiSlices} and \Cref{AmpleSemi} that the compact open sections of $\mathcal{U}(B)$ form a restriction semigroup relative to the compact open subsets of $\mathcal{U}(B)^0$.

\begin{theorem}\label{RSrep}
    The map $b\mapsto O_b$ is an isomorphism of $B$ onto a restriction semigroup formed from a basis of compact open sections of $\mathcal{U}(B)$.
\end{theorem}

\begin{proof}
    By \Cref{StoneRep}, the map $b\mapsto O_b$ is an order isomorphism of $B$ onto a basis of compact open subsets of $\mathcal{U}(B)$.  By \Cref{FilterSource} and \Cref{UltraSource}, the source map on $\mathcal{U}(B)$ is a bijection from each $O_b$ onto $O_{\mathsf{s}(b)}$, i.e.~each $O_b$ is a section.  It only remains to verify that the product is preserved, i.e.~that $O_{ab}=O_aO_b$, for all $a,b\in B$.  Certainly $O_aO_b\subseteq O_{ab}$ so conversely let us take some $U\in O_{ab}$.  Then $\mathsf{s}(\mathsf{s}(a)b)=\mathsf{s}(ab)\in\mathsf{s}(U)$ so \Cref{FilterSource} yields $W=(\mathsf{s}(a)b\mathsf{s}(U))^\leq\in O_{\mathsf{s}(a)b}\subseteq O_b$ with $\mathsf{s}(W)=\mathsf{s}(U)$.  Then $\mathsf{s}(a)\in W^\mathsf{R}\subseteq\mathsf{r}(W)$ so \Cref{FilterSource} again yields $V=(a\mathsf{r}(W))^\leq\in O_a$ with $\mathsf{s}(V)=\mathsf{r}(W)$.  Thus $(V,W)\in\mathcal{U}(B)^2$ and $V\cdot W\in O_aO_b\subseteq O_{ab}$ with $\mathsf{s}(V\cdot W)=\mathsf{s}(W)=\mathsf{s}(U)$ and hence $V\cdot W=U$, as $O_{ab}$ is a section.  This shows that $O_{ab}\subseteq O_aO_b$.
\end{proof}

The next statement characterises when $\mathcal{U}(B)$ is a category.

\begin{proposition}\label{UltraCategory}
    The ultrafilters $\mathcal{U}(B)$ form a category precisely when every element of $B$ is a finite join of elements of $PB$.
\end{proposition}

\begin{proof}
    If every element of $B$ is a finite join of elements of $PB$ then, by \Cref{UltraPrime}, every $U\in\mathcal{U}(B)$ will contain $pb$, for some $p\in P$ and $b\in B$.  This means $p\in U^\mathsf{R}$ and hence $U$ has a range unit $U^{\mathsf{R}\leq}$, by \Cref{AmpleSemi}, showing that $\mathcal{U}(B)$ is a category.

    Conversely, say we have $a\in B$ that is not a join of elements of $PB$.  Thus the order ideal $I$ generated by $Pa$ does not contain $a$ so the set $a\setminus I=\{a\setminus b\mathrel{|}b\in I\}$ does not contain $0$ and hence neither does the filter $(a\setminus I)^\leq$. By the Kuratowski-Zorn lemma, we can extend $a\setminus I$ to an ultrafilter $U$. Thus $U\in\mathcal{U}(B)$ but $U$ can not have a range unit -- if it did then, taking any $p\in U^\mathsf{R}$, we would have $pa\in\mathsf{r}(U)U\subseteq U$ and thus $0 = pa \wedge (a \setminus pa) \in U$. This contradiction shows that $\mathcal{U}(B)$ is not a category.
\end{proof}

Likewise, we can characterise when $\mathcal{U}(B)$ is a groupoid.

\begin{proposition}\label{UltraGroupoid}
    The ultrafilters $\mathcal{U}(B)$ form a groupoid precisely when every element of $B$ is a finite join of invertible elements $B^\times$.
\end{proposition}

\begin{proof}
    This can be proved much like in the proof of \Cref{UltraCategory} but using \Cref{UltraInvertible}.  In particular, if there is some $a \in B$ that is not a join of elements of $B^\times$, then the proof is the same but with $B^\times \cap a^\geq$ instead of $Pa$.  Indeed, if there is a $b \in U \cap B^\times$ in the resulting ultrafilter $U$, then there is a $c \in U$ such that $c \leq a, b$.  Recall from \eqref{IdealInvertibles} that $B^{\times\geq} = B^\times$, so $c \in B^\times \cap a^\geq$, but then $0 = c \wedge (a \setminus c) \in U$.  Hence, $U \cap B^\times = \emptyset$, so $U$ is not invertible by \Cref{UltraInvertible}, meaning $\mathcal{U}(B)$ is not a groupoid.
\end{proof}

\section{Boolean-Cartan Restriction Semigroups}\label{s:BCrs}

In this section, we introduce the notion of a Boolean-Cartan restriction semigroup of an algebra and show the Steinberg algebra of any Hausdorff ample semicategory contains a Boolean-Cartan restriction semigroup (\Cref{Semicategory->BooleanCartan}). We then prove a reconstruction result that every algebra containing a Boolean-Cartan restriction semigroup is isomorphic to the Steinberg algebra of the associated ultrafilter semicategory (\Cref{SteinbergAlgebraRepresentation}).

\begin{definition}\label{def:BC}
    Assume $A$ is a $K$-algebra, for some ring $K$.  Further assume we have a multiplicative subsemigroup $B$ commuting with $K$ which generates $A$ and is a restriction semigroup with respect to a given subset of projections $P$.  We then say $B$ is \emph{Boolean-Cartan} if
    \begin{align}
        \tag{Torsion-Free}\label{TorsionFree}k\in K\setminus\{0\}\quad\text{and}\quad b\in B\setminus\{0\}\qquad&\Rightarrow\qquad kb\neq0,\\
        \tag{Subtractive}\label{Subtractive}p\leq q\in P\qquad&\Rightarrow\qquad q-p\in P,
    \end{align}
    and, moreover, we have an operation $\wedge:A\times B\rightarrow A$ that is linear in $A$ such that
    \[\tag{Meets}\label{Meets}a,b\in B\qquad\Rightarrow\qquad a\wedge b=\max(a^\geq\cap b^\geq).\]
\end{definition}

The examples we have in mind come from Hausdorff ample semicategories $S$.  For any ring $K$, recall that $KS$ denotes the $K$-algebra from \Cref{KSalg} of compactly supported locally constant $K$-valued functions under convolution.

\begin{proposition}\label{Semicategory->BooleanCartan}
    For any ring $K$ and Hausdorff ample semicategory $S$, the characteristic functions $\chi_O$ of compact open sections $O\subseteq S$ form a Boolean-Cartan semigroup in $KS$ with respect to the projections given by characteristic functions of compact open subsets of $S^0$.
\end{proposition}

\begin{proof}
    We denote the relevant sets of characteristic functions $\chi_O$ as follows.
    \begin{align*}
        B&=\{\chi_O\mathrel{|}O\subseteq S\text{ is a compact open section}\}.\\
        P&=\{\chi_O\mathrel{|}O\subseteq S^0\text{ is a compact open set}\}.
    \end{align*}
    First note $B$ commutes with $K$ and satisfies \eqref{TorsionFree} because, for any $k\in K$ and $O\subseteq S$, both $k\chi_O$ and $\chi_Ok$ are $k$ on $O$ and $0$ elsewhere.  Since $f$ is locally constant and compactly supported, there are $\{k_1, \dots, k_n\} \subseteq K \setminus \{0\}$ such that $\mathrm{supp}(f) = \bigcup_{j=1}^n f^{-1}\{k_j\}$.  Because $f^{-1}\{k_j\}$ is clopen for all $j$, by taking relative complements we attain pairwise disjoint compact open sets $N_j$ for all $j$ such that $\mathrm{supp}(f) = \bigcup_{j=1}^n N_j$ and $f$ is $k_j$ on $N_j$.  Thus, $f = \sum_{j=1}^n k_j\chi_{N_j}$.  As $S$ is Hausdorff and ample, every compact open subset of $S$ is a disjoint union of compact open sections.  In particular, this applies to $N_j$ so that $\chi_{N_j}$ is a finite sum of elements of $B$, for all $j$, from which it follows that $B$ generates $KS$.  Also $\chi_{ON}=\chi_O\chi_N$, for any compact open sections $O,N\subseteq S$, so $B$ is a restriction semigroup with respect to the projections $P$, by \Cref{SemiSlices}.  If $O,N\subseteq S^0$ are compact open and $O\subseteq N$ then $N\setminus O$ is also a compact open subset of $S^0$ with $\chi_N-\chi_O=\chi_{O\setminus N}$, i.e.~\eqref{Subtractive} also holds.  Finally, let $\wedge$ denote the pointwise product $KS\times B \to KS$, i.e.~for any $a\in KS$ and $b\in B$ we define $a\wedge b:S\rightarrow K$ by
    \[(a\wedge b)(s)=a(s)b(s).\]
    As $a$ and $b$ are locally constant, so is $a\wedge b$.  Also $\mathrm{supp}(a\wedge b)=\mathrm{supp}(a)\cap\mathrm{supp}(b)$ is a compact open section so $a\wedge b$ is indeed an element of $KS$.  Because $B$ commutes with $K$, $a\mapsto a\wedge b$ is (left- and right-)linear, for all $b\in B$.  That is, $KS \times B \to KS$ is linear in $KS$.  (Note that $KS \times B \to KS$ is right-linear in $B$ but can fail to be left-linear in $B$ because $KS$ may not commute with $K$.)  Also $\chi_O\wedge\chi_N=\chi_{O\cap N}$, for all $O,N\subseteq B$, so $\wedge$ is indeed the meet operation on $B$.  Thus $B$ is a Boolean-Cartan semigroup in $KS$.
\end{proof}

Above it would have sufficed to restrict to a basis of compact open sections such that $O\setminus N$ and $\mathsf{s}[O]$ lie in the basis whenever $O$ and $N$ do.  Our goal is to show these are, up to isomorphism, the only examples, and thus Steinberg algebras of Hausdorff ample semicategories are completely characterised by having Boolean-Cartan restriction subsemigroups.

To avoid repeating ourselves let us assume from now on that
\begin{center}
    \textbf{$K$ is a ring, $A$ is a $K$-algebra and $B\subseteq A$ is a Boolean-Cartan restriction semigroup with respect to some given projections $P\subseteq B$.}
\end{center}

The first thing to observe is that Boolean-Cartan semigroups are indeed semi-Boolean.

\begin{proposition}\label{SemiBooleanRestriction}
    The Boolean-Cartan $B$ is a semi-Boolean restriction semigroup with 
    \begin{equation}\label{SumsMeetsJoins}
        b+c=b\vee c\quad\text{and}\quad a\wedge(b+c)=a\wedge b+a\wedge c
    \end{equation}
    for all $a \in A$ and disjoint bounded $b,c\in B$.
\end{proposition}

\begin{proof}
    Certainly $B\neq\emptyset$, as $B$ spans $A$, and hence $P\neq\emptyset$ in order for elements of $B$ to have source projections.  Thus we can take $p\in P$ and note that $0=p-p\in P$, by \eqref{Subtractive}.  Thus $0=sp-sp = s(p-p) = s0$ for all $s\in B$ and $p\in P$, so $0$ is the minimum of $B$.  Also \eqref{Meets} means that $B$ is a meet-semilattice.

    Let $c \in B$.  We show $c^\geq$ is a Boolean algebra.  For any $b \leq c$, we have $c-b=c\mathsf{s}(c)-c\mathsf{s}(b)=c(\mathsf{s}(c)-\mathsf{s}(b)) \in BP \subseteq B$, using \eqref{Subtractive}.  Recall from \eqref{ConditionalMeets} that if $a,b\leq c$ then $a\wedge b=b\mathsf{s}(a)$.  In particular, $(c-b)\mathsf{s}(b) = c\mathsf{s}(b)-b\mathsf{s}(b)=b-b=0$, so $c-b \in b^\top$.  Moreover, if $a\wedge b=0$ then $(c-b)\mathsf{s}(a)=c\mathsf{s}(a)-b\mathsf{s}(a)=a$, and so $a\leq c-b$.  This shows that $c-b$ is a pseudocomplement of $b$ in $c^\geq$.  As $c-(c-b)=b$, it follows that $c^\geq$ is a Boolean algebra, by \Cref{Byrne}, showing that $B$ is indeed a semi-Boolean algebra. 
    
    Still assuming $a,b\leq c$ and $a\wedge b=0$, it follows that $a\vee b=c\setminus((c\setminus a)\setminus b)=c-((c-a)-b)=a+b$.  Also, $ad,bd\leq cd$ and $ad\wedge bd=(a\wedge b)d=0$ by \eqref{MeetPreservingRight}, so $ad\vee bd=ad+bd$.  Altogether, we have $(a\vee b)d=(a+b)d=ad+bd=ad\vee bd$.  Even when $a\wedge b\neq0$, we then see that $(a\vee b)d=(a\vee(b\setminus a))d=ad\vee(b\setminus a)d\leq ad\vee bd\leq(a\vee b)d$.  This shows that $B$ is indeed a semi-Boolean restriction semigroup.

    Lastly, assume $a\in A$ and $b,c\in B$ are disjoint and bounded.  Since $\wedge$ is linear in $A$ and $B$ spans $A$, it suffices to check $a\wedge(b+c)=a\wedge b+a\wedge c$ when $a \in B$, in which case $a\wedge(b+c)=a\wedge(b\vee c)=a\wedge b\vee a\wedge c=a\wedge b+a\wedge c$.
\end{proof}

The following is then immediate from \Cref{AmpleSemi}.

\begin{corollary}
    The ultrafilter semicategory $\mathcal{U}(B)$ is a Hausdorff ample semicategory.
\end{corollary}

As $A=\mathrm{span}(B)$, every $a\in A$ can be expressed as a finite linear combination of elements of $B$.  We can then use relative complements to turn it into a linear combination of disjoint elements of $B$, i.e.~$a=k_1b_1+\ldots+k_nb_n$ where $b_i\wedge b_j$ whenever $i\neq j$.  Such combinations will not be unique, but can still be used to define linear functionals from ultrafilters as follows.  Also note that $(a\wedge b)\wedge c=a\wedge (b\wedge c)$ holds for all $a,b,c\in B$ and thus extends even to $a\in A$, as $\wedge$ is linear in $A$ and $A=\mathrm{span}(B)$.  Thus we may just write $a\wedge b\wedge c$.
    
\begin{proposition}
    For each ultrafilter $U\subseteq B$, the characteristic function $\chi_U$ of $U$ on $B$ $($with $\chi_U[U]=\{1\}$ and $\chi_U[B\setminus U]=\{0\})$ extends uniquely to a linear map from $A$ to $K$.
\end{proposition}

\begin{proof}
    Uniqueness follows immediately from linearity and the fact $B$ spans $A$.  For existence we first claim that, for every $a\in A$, we have some $u\in U$ and $k\in K$ with $a\wedge u=ku$.  To see this, express $a$ as a disjoint linear combination of elements of $B$, say $a=k_1b_1+\ldots+k_nb_n$.  If we have $u\in U$ with $b_j\wedge u=0$, for all $j$, then $a\wedge u=\sum_{j=1}^nk_j(b_j\wedge u)=0=0u$ and we are done.  Otherwise, for any $u\in U$, we must have $\bigwedge_{j=1}^nu\setminus b_j\notin U$ and hence $u\setminus(\bigwedge_{j=1}^nu\setminus b_j)\in U$, as $U$ is an ultrafilter.  Replacing $u$ with this latter element if necessary, we may thus assume that $\bigwedge_{j=1}^nu\setminus b_j=0$ and hence $u=\bigvee_{j=1}^n u\wedge b_j\in U$.  As $U$ is an ultrafilter, we must then have $u\wedge b_i\in U$, for some $i$, and hence
    \[a\wedge u\wedge b_i=\left(\sum_{j=1}^nk_jb_j\right)\wedge u\wedge b_i=k_i(u\wedge b_i).\]
    This completes the proof of the claim.  Also note that the $k$ here is unique.  Indeed, if $a\wedge t=kt$ and $a\wedge u=lu$, for some $t,u\in U$, then $t\wedge u\in U$ and $k(t\wedge u)=a\wedge t\wedge u=l(t\wedge u)$ so $(k-l)(t\wedge u)=0$ and hence $k-l=0$, by \eqref{TorsionFree}, i.e.~$k=l$.

    Thus we have a unique function $f:A\rightarrow K$ such that, for every $a\in A$, we have some $u\in U$ with $a\wedge u=f(a)u$.  For any $k\in K$, note that then $ka\wedge u=k(a\wedge u)=kf(a)u$ and hence $f(ka)=kf(a)$.  On the other hand, for any other $a'\in A$, we have $u'\in U$ with $a\wedge u'=f(a')u'$.  Setting $t=u\wedge u'\in U$, it follows that
    \[a\wedge t=a\wedge u\wedge u'=f(a)u\wedge u'=f(a)(u\wedge u')=f(a)t.\]
    Likewise, $a'\wedge t=f(a')t$ and hence
    \[(a+a')\wedge t=a\wedge t+a'\wedge t=f(a)t+f(a')t=(f(a)+f(a'))t.\]
    This shows that $f(a+a')=f(a)+f(a')$, i.e.~$f$ is linear.  Lastly note that if $u\in U$ then $u\wedge u=u=1u$ so $f(u)=1$, i.e.~$f[U]=\{1\}$, while if $b\in B\setminus U$ then we have $u\in U$ with $b\wedge u=0=0u$ and hence $f(b)=0$, i.e.~$f[B\setminus U]=\{0\}$. Thus $f = \chi_U$ on $B$, as desired.
\end{proof}

Denote the linear extension above again by $\chi_U$.  Given $a\in A$, define $\hat{a}:\mathcal{U}(B)\rightarrow K$ by
\[\hat{a}(U)=\chi_U(a).\]
For any $b\in B$, note $\hat{b}$ is the characteristic function $\chi_{O_b}$ of $O_b\subseteq\mathcal{U}(B)$, as
\[\chi_{O_b}(U)=1\qquad\Leftrightarrow\qquad U\in O_b\qquad\Leftrightarrow\qquad b\in U\qquad\Leftrightarrow\qquad\hat{b}(U)=\chi_U(b)=1.\]

\begin{theorem}\label{SteinbergAlgebraRepresentation}
    The map $a\mapsto\hat{a}$ is an isomorphism of $A$ onto $K\mathcal{U}(B)$, the algebra of compactly supported locally constant $K$-valued functions on the ultrafilter semicategory $\mathcal{U}(B)$.
\end{theorem}

\begin{proof}
    As $\chi_U$ is linear, for all $U\in\mathcal{U}(B)$, so is the map $a\mapsto\hat{a}$.  To show $a\mapsto\hat{a}$ also preserves products it suffices to consider elements of $B$, as $B$ generates $A$.  By \Cref{RSrep}, $O_{ab}=O_aO_b$ and hence
    \[\widehat{ab}=\chi_{O_{ab}}=\chi_{O_aO_b}=\chi_{O_a}\chi_{O_b}=\hat{a}\hat{b},\]
    for any $a,b\in B$.  Thus products are preserved and $a\mapsto\hat{a}$ is an algebra homomorphism.
    
    Now any $a\in A\setminus\{0\}$ can be written as $a=\sum_{j=1}^nk_jb_j$ for non-zero $k_1,\ldots,k_n\in K$ and disjoint $b_1,\ldots,b_n\in B$, where $n\geq1$.  Taking any $U\in \mathcal{U}(B)$ containing some $b_j$, it follows that $\hat{a}(U)=k_j\neq0$, showing that $a\mapsto\hat{a}$ is injective.  On the other hand, using similar arguments as in the proof of \Cref{Semicategory->BooleanCartan} and the basis $(O_a)_{a\in B}$ identified in \Cref{RSrep}, every compactly supported locally constant function on $\mathcal{U}(B)$ can be expressed as a linear combination of $\hat{a}=\chi_{O_a}$ for $a \in B$, showing $a\mapsto\hat{a}$ is surjective.
\end{proof}

Steinberg algebras were originally defined in \cite{Steinberg2010,CFST14} over groupoids rather than the more general semicategories considered here.  \Cref{SteinbergAlgebraRepresentation} specialises to the groupoid case precisely when every element of $B$ is a finite join of elements of $B^\times$, by \Cref{UltraGroupoid}.  As these finite joins can be made disjoint and thus the same as finite sums, this means that $B^\times$ will also generate $A$ and hence be a Boolean-Cartan semigroup in its own right.  With such Boolean-Cartan inverse semigroups, we can replace the $\wedge$ operation on $A\times B$ with a single expectation $\Phi$ on $A$, more like in the usual theory of Cartan semigroups/subalgebras of Steinberg algebras and operator algebras -- see \cite{9a23} and \cite{BiceClarkYingFenMcCormick2025}.

\begin{proposition}\label{Expectation}
    Assume $I\subseteq A$ is an inverse semigroup spanning $A$ and commuting with $K$ satisfying \eqref{TorsionFree} and \eqref{Subtractive} with $I$ and its idempotents $E$ in place of $B$ and $P$ respectively.  Then $I$ is Boolean-Cartan precisely when we have a linear map $\Phi:A\rightarrow A$ such that, for all $b\in I$,
    \[\Phi(b)=\max(b^\geq\cap E).\]
\end{proposition}

\begin{proof}
    Given such a linear map $\Phi$ on $A$, we define $\wedge$ on $A\times I$ by
    \[a\wedge b=\Phi(ab^{-1})b.\]
    Because $\Phi$ is linear and $I$ commutes with $K$, $\wedge$ is linear in $A$.
    This $\wedge$ operation coincides with the fixed point operator of $I$ in the sense of \cite{Leech1995} for inverse monoids or \cite{Lawson2023} for general inverse semigroups.  Hence, $\Phi$ satisfies \eqref{Meets} by \cite[Theorem 1.9(a)]{Leech1995} or \cite[Lemma 3.2(4)]{Lawson2023}.

    Conversely, say we have an operation $\wedge:A\times I\to A$ that is linear in $A$ and satisfies \eqref{Meets}.  If $I$ is a monoid with unit $1$ then we can simply define $\Phi(a)=a\wedge 1$.  In general we proceed as follows.  For every finite $F\subseteq E$, let us define $\Phi_F:A\rightarrow K$ by
    \[\Phi_F(a)=\sum_{G\subsetneqq F}a\wedge\bigwedge_{f\in F\setminus G,g\in G}f\setminus g\]
    (where we take $\bigwedge_{f\in F\setminus G,g\in G}f\setminus g$ to mean $\bigwedge_{f\in F}f$ when $G=\emptyset$).  We aim to show the net $(\Phi_F)$, indexed by finite $F\subseteq E$ ordered by inclusion, converges pointwise, i.e.~for all $a \in A$, the net $(\Phi_F(a))$ is eventually constant with respect to $F$.  Fix $a \in A$.  As $I$ spans $A$, we can write $a$ as a linear combination $a=\sum_{j\leq m}k_jb_j$ of elements of $I$.  Write $C=\{b_1,\ldots,b_n\}$.  We claim $\Phi_F(a)=\Phi_{\mathsf{s}[C]}(a)$ for all finite $F\subseteq E$ containing $\mathsf{s}[C]$ (in particular, this means the limit is independent of the chosen $C$ spanning $a$).  It suffices to show $\Phi_F(a)=\Phi_{F\setminus\{f\}}(a)$ for all finite $F\subseteq E$ containing $\mathsf{s}[C]$ and for all $f \in F\setminus\mathsf{s}[C]$.  Fix such an $F$ and $f \in F\setminus\mathsf{s}[C]$, and write $F=\{e_1,\ldots,e_n,f\}$.  First we identify a redundant summand in $\Phi_F(a)$.  The summand of $\Phi_F(a)$ corresponding to $\{e_1,\ldots,e_n\}\subsetneqq F$ is $a\wedge \bigwedge_{i\leq n}f\setminus e_i=\sum_{j\leq m}k_j(b_j\wedge\bigwedge_{i\leq n}f\setminus e_i)$ because $\wedge$ is linear in $A$.  For each $j\leq m$, $\mathsf{s}(b_j)\in\mathsf{s}[C]\subseteq\{e_1,\ldots,e_n\}$, so $b_j\wedge\bigwedge_{i\leq n}f\setminus e_i\leq b_j\wedge f\setminus \mathsf{s}(b_j)$.  Because $f\in E$, $b_j\wedge f\setminus \mathsf{s}(b_j)=0$.  Hence, the summand of $\Phi_F(a)$ corresponding to $\{e_1,\ldots,e_n\}\subsetneqq F$ is $0$.  Now, each proper subset $G$ of $\{e_1,\ldots,e_n\}$ corresponds with the pair $(G,G\cup\{f\})$ of proper subsets of $F$, which index the remaining summands of $\Phi_{\{e_1,\ldots,e_n\}}(a)$ and $\Phi_F(a)$, respectively.  Thus, given $G\subsetneqq\{e_1,\ldots,e_n\}$, we just need
    \[a\wedge\bigwedge_{e\in\{e_1,\ldots,e_n\}\setminus G,g\in G}e\setminus g=\left(a\wedge\bigwedge_{e\in F\setminus G,g\in G}e\setminus g\right)+\left(a\wedge\bigwedge_{e\in F\setminus G\cup\{f\},g\in G\cup\{f\}}e\setminus g\right),\]
    which follows from \eqref{SumsMeetsJoins}, and so the claim holds.

    Thus $(\Phi_F)$ converges pointwise.  The limit $\Phi:A\rightarrow K$ satisfies $\Phi(a)=\Phi_{\mathsf{s}[C]}(a)$ for any finite $C\subseteq I$ with $a\in\mathrm{span}(C)$.  It follows that $\Phi$ is linear.  Moreover, for all $b\in I$, one sees that $\Phi(b)=\Phi_{\{\mathsf{s}(b)\}}(b)=b\wedge\mathsf{s}(b)=\max\{p\in E\mathrel{|}p\leq b\}$, as required.
\end{proof}

As an example, one could take $S$ to be the one element groupoid, in which case $KS$ can be identified with $K$ viewed as algebra over itself.  Then $B$ and $P$ above get identified with the subsemigroup $\{0,1\}$ of $K$.  Note that the larger subsemigroup $\{-1,0,1\}$ would still satisfy all the required conditions except for \eqref{Meets}, as
\[(-1)\wedge1=0\neq-1=-(1\wedge 1),\]
showing that the meet operation on $\{-1,0,1\}$ does not extend linearly to the algebra $K$.

Specialising even further, we recover a classic result of Keimel from \cite{Keimel1970}.  Our work thus constrasts with previous reconstruction results, e.g.~from \cite{9a23}, which usually assume Keimel's result from the outset and build up from there.

\begin{proposition}\label{Keimel}
    If the idempotents $E$ of $A$ span $A$, commute with each other and $K$ and satisfy \eqref{TorsionFree} then $A$ is isomorphic to the algebra of compactly supported locally constant $K$-valued functions on a locally compact Hausdorff $0$-dimensional space.
\end{proposition}

\begin{proof}
    If $p,q\in E$ and $p\leq q$ then $(q-p)(q-p)=q-p-p+p=q-p$ so $q-p\in E$ too, showing $E$ also satisfies \eqref{Subtractive}.  We can also just take $\wedge:A\times E\rightarrow A$ to be the product, i.e.~$a\wedge p=ap$, so that $\wedge$ is linear in $A$ and satisfies \eqref{Meets}.  Thus $E$ is a Boolean-Cartan semigroup consisting entirely of projections.  By \eqref{Units}, every ultrafilter is then a unit so $\mathcal{U}(E)$ has a trivial product.  Thus $A$ is isomorphic to the algebra of compactly supported locally constant $K$-valued functions on the ultrafilter space $\mathcal{U}(E)$, by \Cref{SteinbergAlgebraRepresentation}, which is locally compact, Hausdorff and $0$-dimensional, by \Cref{T1,0dim,LocallyCompact}.
\end{proof}

\section{Semigroup Algebra Quotients}\label{s:8}

Here we show how to view an algebra with Boolean-Cartan semigroup $B$ as a quotient of the semigroup algebra of $B$.  First let us assume throughout this section that
\begin{center}
    \textbf{$K$ is a ring and $B$ is a semi-Boolean restriction semigroup with projections $P$.}
\end{center}

Letting $B_0=B\setminus\{0\}$, note that the finitely supported functions from $B_0$ to $K$ form a $K$-bimodule under the usual pointwise operations.  For all $a,b\in B_0$, define
\[\delta_b(a)=\begin{cases}1&\text{if }a=b\\0&\text{otherwise}.\end{cases}\]
We also define $\delta_0(a)=0$, for all $a\in B_0$.  We turn the space of finitely supported $K$-valued functions on $B_0$ into an algebra $KB_0$ by specifying that, for all $a,b\in B_0$,
\[\delta_a\delta_b=\delta_{ab}.\]
If $K$ is commutative, $KB_0$ coincides with the contracted semigroup algebra $K_0B$ in \cite[\S2.4]{SteinbergSzakacs2021}.

As $a=0$ precisely when $\mathsf{s}(a)=0$, and $\mathsf{s}(a\mathsf{s}(b))=\mathsf{s}(\mathsf{s}(a)\mathsf{s}(b))=\mathsf{s}(\mathsf{s}(b)\mathsf{s}(a))=\mathsf{s}(b\mathsf{s}(a))$, so
\[\tag{Orthogonality}a\mathsf{s}(b)=0\qquad\Leftrightarrow\qquad\mathsf{s}(a)\mathsf{s}(b)=0\qquad\Leftrightarrow\qquad b\mathsf{s}(a)=0.\]
In this case we call $a,b\in B$ \emph{orthogonal}.  Any orthogonal $a,b\in B$ are also disjoint.  Indeed, if $c\leq a,b$, then $c=a\mathsf{s}(c)\leq a\mathsf{s}(b)=0$.

\begin{proposition}
    We have an ideal in $KB_0$ defined by
    \[I=\mathrm{span}\{\delta_a+\delta_b-\delta_{a\vee b}\mathrel{|}a,b\in B_0\text{ are orthogonal and bounded}\}.\]
\end{proposition}

\begin{proof}
    For all orthogonal bounded $a,b\in B_0$ and $c\in B_0$, note $(\delta_a+\delta_b-\delta_{a\vee b})\delta_c=\delta_{ac}+\delta_{bc}-\delta_{ac\vee bc}\in I$, as $ac\mathsf{s}(bc)=a\mathsf{s}(b)c=0$.  Similarly, $\delta_c(\delta_a+\delta_b-\delta_{a\vee b})=\delta_{ca}+\delta_{cb}-\delta_{ca\vee cb}\in I$ because $ca\mathsf{s}(cb)\leq ca\mathsf{s}(b)=0$.  As the $(\delta_c)_{c\in B_0}$ span $KB_0$, this shows that $I$ is an ideal.
\end{proof}

Next note we have a map $\wedge:KB_0\times B\rightarrow KB_0$, which is linear in $KB_0$, defined by
\[\delta_a\wedge b=\delta_{a\wedge b}.\]
We then immediately see that, for all orthogonal bounded $a,b\in B_0$ and $c\in B$,
\[(\delta_a+\delta_b-\delta_{a\vee b})\wedge c=\delta_{a\wedge c}+\delta_{b\wedge c}-\delta_{(a\vee b)\wedge c}=\delta_{a\wedge c}+\delta_{b\wedge c}-\delta_{(a\wedge c)\vee(b\wedge c)}\in I.\]
Thus $\wedge$ induces a map on $KB_0/I\times B$, which we again denote by $\wedge$, so for all $f\in KB_0$,
\[(f+I)\wedge b=(f\wedge b)+I\]

Next we note that the functions $(\delta_b)_{b\in B_0}\subseteq KB_0$ remain distinct in $KB_0/I$.

\begin{proposition}\label{IdealElementSupport}
    If $f\in I$ then $|\mathrm{supp}(f)|\neq1,2$.
\end{proposition}

\begin{proof}
    First we claim that, for all $a,b,c\in B_0$ with $a$ and $b$ orthogonal and bounded, we have $d\in B_0$ with $d\leq c$ and $(\delta_a+\delta_b-\delta_{a\vee b})\delta_d=0$.  Indeed, note that this last equation holds whenever $d=\mathsf{s}(a)c$, $\mathsf{s}(b)c$ or $c\setminus(\mathsf{s}(a)c\vee\mathsf{s}(b)c)$.  As $c$ is the join of these three elements, one of them must be non-zero, proving the claim.  Then $(\delta_{a}+\delta_{b}-\delta_{a\vee b})\delta_{d'}=0$ too, for any $d'\leq d$.  So by induction the same result then holds for spans, i.e.~for all $f\in I$ and $c\in B_0$, we have $d\in B_0$ with $d\leq c$ and $f\delta_d=0$.  From this it follows that if $k\delta_c\in I$ then $k=0$, as we can take $p\in B_0$ with $p\leq\mathsf{s}(c)$ and $k\delta_c\delta_p=0$ so $cp\neq0$ and $k\delta_{cp}=k\delta_c\delta_p=0$.  In other words, if $f\in I$ then $|\mathrm{supp}(f)|\neq1$.

    Now say $k\delta_a+l\delta_b\in I$, for some $k,l\in K$ and distinct $a,b\in B_0$.  Then $a\setminus b\neq0$ and
    \[k\delta_{a\setminus b}=(k\delta_{a\wedge(a\setminus b)}+l\delta_{b\wedge(a\setminus b)})=(k\delta_a+l\delta_b)\wedge(a\setminus b)\in I.\]
    Thus $k=0$ and, likewise $l=0$, showing that $|\mathrm{supp}(f)|\neq2$ whenever $f\in I$.
\end{proof}

Every semi-Boolean restriction semigroup arises as a Boolean-Cartan semigroup of a $K$-algebra, for any given ring $K$, thanks to \Cref{RSrep}, \Cref{Semicategory->BooleanCartan} and the comment immediately after.  Now we can also prove this more directly as follows.

\begin{corollary}\label{QuotientBCsemigroup}
    $KB_0/I$ has a Boolean-Cartan semigroup isomorphic to $B$ given by
    \[B/I=\{\delta_b+I\mathrel{|}b\in B\}.\]
\end{corollary}

\begin{proof}
    By \Cref{IdealElementSupport}, $\delta_a-\delta_b\notin I$, for all distinct $a,b\in B$, so $B/I$ is isomorphic to $B$.  In particular, it is a semi-Boolean restriction semigroup generating $KB_0/I$.  Now if $p,q\in P$ and $p\leq q$ then $p$ and $q\setminus p$ are orthogonal with join $q$ and hence $\delta_p+\delta_{q\setminus p}-\delta_q\in I$ or, equivalently, $\delta_q-\delta_p+I=\delta_{q\setminus p}+I\in P/I=\{\delta_p+I\mathrel{|}p\in P\}$, showing that \eqref{Subtractive} holds.  Also, again by \Cref{IdealElementSupport}, $k\delta_b\in I$ implies $k=0$ or $b=0$ so \eqref{TorsionFree} holds.  Finally, identifying $B$ with $B/I$ we obtain a map $\wedge:KB_0\times B/I\rightarrow KB_0$ which is linear in $KB_0$ and extends the meet operation on $B/I$.
\end{proof}

From this we obtain a result analogous to \cite[Corollary 2.14]{SteinbergSzakacs2021}.

\begin{corollary}
    The semigroup algebra quotient $KB_0/I$ is isomorphic to the algebra $K\mathcal{U}(B)$ of compactly supported locally constant $K$-valued functions on the ultrafilter semicategory $\mathcal{U}(B)$.
\end{corollary}

\begin{proof}
    Immediate from \Cref{SteinbergAlgebraRepresentation} and \Cref{QuotientBCsemigroup}.
\end{proof}

As a converse to \Cref{QuotientBCsemigroup}, we also have the following.

\begin{proposition}\label{prop:iso}
    If a $K$-algebra $A$ has a Boolean-Cartan semigroup $B$ then $KB_0/I$ is isomorphic to $A$ via the unique algebra homomorphism sending each $\delta_b+I$ to $b$.
\end{proposition}

\begin{proof}
    The algebra homomorphism $\pi$ from $KB_0$ to $A$ sending each $\delta_b$ to $b$ has $I$ in its kernel.  Indeed, for all orthogonal bounded $a,b\in B$, $a$ and $b$ are also disjoint, so $a\vee b=a+b$ by \eqref{SumsMeetsJoins}.  It suffices to show that the kernel of $\pi$ contains no other elements of $KB_0$.  To see this, say we had $f\in KB_0\setminus I$ with $\pi(f)=0$.  Adjusting by elements of $I$ if necessary, we can turn $f$ into an element $g\in KB_0\setminus I$ of the form $g=k_1\delta_{a_1}+\ldots+k_n\delta_{a_n}$ for disjoint $a_1,\ldots,a_n\in B_0$ and $k_1,\ldots,k_n\in K$.  Then $0=\pi(g)=k_1a_1+\ldots+k_na_n$ so, for each $j\leq n$, we also see that $0=\pi(g)\wedge a_j=ka_j$ and hence $k_j=0$.  As $j$ was arbitrary, this means $0=g\notin I$, a contradiction.
\end{proof}

\section{Pseudo-Cartan, quasi-Cartan, Cartan and diagonal subalgebras}\label{s:9}

In this section, we relate our Boolean-Cartan semigroups with various commutative subalgebras from \cite{9a23}, where the ring $K$ is assumed to be commutative and \emph{indecomposable}, i.e.~the only idempotents are $0$ and $1$.
Specifically, given a Boolean-Cartan inverse semigroup, we find conditions that give rise to a natural pseudo-Cartan subalgebra.
In \Cref{QuasiCartan}, we show that this subalgebra is quasi-Cartan under an analogue of the local bisection hypothesis of \cite{Ste19}.
In this case, we show the ultrafilter groupoid of the Boolean-Cartan inverse semigroup is isomorphic to the base groupoid of the discrete twist used to reconstruct the algebra in \cite{9a23} (\Cref{GpdIsom}).
As a corollary, and drawing on results from \cite{9a23,Lawson2023}, we characterise when the subalgebra is a Cartan or diagonal subalgebra in terms of the Boolean-Cartan inverse semigroup (\Cref{CartanDiagonal}).

Throughout this section we assume that

\begin{center}
    \textbf{$A$ is an algebra over an indecomposable commutative ring $K$.}
\end{center}

\begin{definition}
    Let $D$ be a commutative subalgebra of $A$. 
    Suppose $D = \mathrm{span}(E(D))$, where $E(D)$ is the set of idempotents of $D$, and that $D$ satisfies 
    \begin{align}
    \tag{Without Torsion}\label{WithoutTorsion}k\in K\setminus\{0\}\text{ and }e\in E(D)\setminus\{0\}\quad&\Rightarrow\quad ke\neq0,\qquad\text{and}\\
    \tag{Local Units}\label{LocalUnits}a\in A\quad&\Rightarrow\quad\exists e \in E(D)\,(ae=a=ea).
    \end{align}
    Further suppose there is an \emph{expectation}, i.e. a linear map $\Phi \colon A \to D$ such that $\Phi(d)=d$ and $\Phi(dad') = d\Phi(a)d'$ for all $d, d' \in D$ and $a \in A$.
    Lastly, suppose $A = \mathrm{span}(N)$, where 
    \[
        N\coloneqq\{n\in A\mathrel{|}\exists m\in A\text{ such that }mnm=n,nmn=n,nDm\cup mDn\subseteq D\}
    \]
    is the set of \emph{normalisers} of $D$ in $A$.
    In the above setting, we say $D$ is a \emph{pseudo-Cartan} subalgebra of $A$.
\end{definition}


For any pseudo-Cartan subalgebra $D$, by \cite[Lemma 2.15]{9a23}, $N$ is an inverse semigroup with inverse $n \mapsto n^{\dagger}$ and idempotents $E(N) = E(D)$, where $n^{\dagger}$ is the unique element of $N$ satisfying $nn^{\dagger}n = n$, $n^{\dagger}nn^{\dagger} = n^{\dagger}$ and $nDn^{\dagger} \cup n^{\dagger}Dn \subseteq D$.
We also assume

\begin{center}
    \textbf{$I$ is a Boolean-Cartan inverse semigroup in $A$ relative to its idempotents $E(I)$, and $E(I)$ is a join-semilattice.}
\end{center}

The \eqref{LocalUnits} condition agrees with that in \cite{Kudryavtseva2025}, which immediately extends to finite subsets, e.g.~if $a,b\in A$ and $c,d\in E(D)$ satisfy $ac=a=ca$ and $bd=b=db$ then $e:=c+d-cd=c+d-dc\in E(D)$, $ae=a=ea$ and $be=b=eb$.  Every Boolean-Cartan semigroup is a conditional join-semilattice, but its projections need not be a join-semilattice, hence the assumption that $E(I)$ is a join-semilattice.

\begin{proposition}\label{PseudoCartan}
    The subalgebra \[D\coloneqq\mathrm{span}(E(I))\] is pseudo-Cartan, $E(D)=E(I)=E(N)\eqqcolon E$, and $I$ is an inverse subsemigroup of $N$.
\end{proposition}
\begin{proof}
    Firstly, note that $D$ is a commutative subalgebra of $A$ because $K$ is commutative and $E(I)$ is a commutative semigroup. Because of \eqref{TorsionFree}, \eqref{WithoutTorsion} will hold provided $E(D)\subseteq I$.
    In fact, we can show $E(D) = E(I)$. 
    Fix $e \in E(D) \setminus \{ 0 \}$.
    We write $e=\sum_{i\leq n}k_ie_i$ for some $k_i\in K\setminus\{0\}$ and $e_i\in E(I)\setminus\{0\}$.
    Then, because $e$ is idempotent, we have $\sum_{i\leq n}k_ie_i = \sum_{i,j\leq n}k_ik_je_ie_j = \sum_{i\leq n}k_i^2e_i$.
    Fix $j\leq n$.
    Using the operation $\wedge:A\times I\to A$, we have $(\sum_{i\leq n}k_ie_i)\wedge e_j=(\sum_{i\leq n}k_i^2e_i)\wedge e_j$, so $k_je_j=k_j^2e_j$.
    Since $e_j \in I\setminus\{0\}$, \eqref{TorsionFree} implies $k_j=k_j^2$.
    Because $K$ is indecomposable and $k_j\neq0$, we have $k_j=1$ for all $j\leq n$.
    Thus, $e=\sum_{i\leq n}e_i$.
    Becuase $E(I)$ is a join-semilattice, we have $\bigvee_{i\leq n}e_i\in E(I)$.
    Then, as in the proof of \Cref{SemiBooleanRestriction}, $e=\sum_{i\leq n}e_i=\bigvee_{i\leq n}e_i \in E(I)$.
    That is, $E(D)\subseteq E(I)$, and the reverse inclusion follows from the definition of $D$, so $E(D)=E(I)\eqqcolon E$.
    Consequently, \eqref{WithoutTorsion} holds.
    Moreover, because $E(D)=E(I)$, $D=\mathrm{span}(E(D))$.
    
    We show \eqref{LocalUnits}, i.e.~$a \in Ea \cap aE$, for all $a \in A$. 
    If $a=\sum_{i=1}^n k_{i}b_{i}$ with $\{b_1,\ldots,b_n\}\subseteq I$, then we let $e\coloneqq\bigvee_{i=1}^n b_{i}b_{i}^{-1}$, using that $E$ is a join-semilattice, and so
    \[
        ea 
        = \sum_{i=1}^n k_{i}eb_{i}
        = \sum_{i=1}^n k_{i}eb_{i}b_{i}^{-1}b_{i}
        = \sum_{i=1}^n k_{i}b_{i}b_{i}^{-1}b_{i}
        = \sum_{i=1}^n k_{i}b_{i}
        = a.
    \]
    Similarly, $f \coloneqq \bigvee_{i=1}^n b_{i}^{-1}b_{i}$ satisfies $af =a$, and so $a \in Ea \cap aE$, as needed.

    Recall from \Cref{Expectation} that the operation $\wedge:A\times I\to A$ induces a linear map $\Phi:A\rightarrow A$ such that, $\Phi(b)=\max(b^\geq\cap E)\in E$ for all $b\in I$.
    That is, $\Phi|_I$ is the fixed point operator of $I$.
    Because $A=\mathrm{span}(I)$ and $D=\mathrm{span}(E)$, we have $\Phi(A)\subseteq D$.
    Also, $\Phi(e)=\max(e^\geq\cap E)=e$ for all $e\in E$, so $\Phi(d)=d$ for all $d\in D$.
    To see that $\Phi$ is an expectation, it remains to show $d\Phi(a)d'=\Phi(dad')$ for all $d,d'\in D$ and $a \in A$.
    Again, $A=\mathrm{span}(I)$ and $D=\mathrm{span}(E)$, so it suffices to show $e\Phi(b)e'=\Phi(ebe')$ for all $b\in I$ and $e,e'\in E$.
    This follows from the identities $e\Phi(b)=\Phi(eb)$ and $\Phi(be)=\Phi(b)e$ for all $b \in I$ and $e \in E$ as in \cite[Lemma 3.2]{Lawson2023}.
    Therefore, $\Phi$ is an expectation. 

    To see $D$ is pseudo-Cartan, it remains to show $A = \mathrm{span}(N)$.
    Because $A = \mathrm{span}(I)$, it suffices to show $I \subseteq N$.
    In fact, we show $I$ is an inverse subsemigroup of $N$. 
    For each $b \in I$, $bb^{-1}b = b$ and $b^{-1}bb^{-1} = b^{-1}$ because $I$ is an inverse semigroup. 
    Moreover, $bDb^{-1}\cup b^{-1}Db\subseteq D$ because $D=\mathrm{span}(E)$. 
    Hence, $I\subseteq N$ and $b^\dagger=b^{-1}$ for all $b\in I$.
\end{proof}

An expectation $\Phi$ is \emph{implemented by idempotents} if, for every $n \in N$, there is some $e \in E(D)$ such that $\Phi(n) = en = ne$.
Like in \cite[Definition 3.3]{9a23}, a pseudo-Cartan subalgebra $D$ with expectation $\Phi$ is \emph{quasi-Cartan} if $\Phi$ is implemented by idempotents.

\begin{remark}
    If an expectation $\Phi$ is implemented by idempotents, then it is automatically \emph{faithful} (i.e. for every $a \in A \setminus \{ 0 \}$, there is some $m \in N$ such that $\Phi(ma) \neq 0$), as per \cite[Proposition 11.3]{Bice2023}.
    Hence, the above definition of quasi-Cartan subalgebras coincides with that of \cite{9a23}.
\end{remark}

By \cite[Theorem 8.7]{9a23}, whenever $D$ is quasi-Cartan, the discrete twist $\Sigma$ over $G$ that is used to reconstruct $A$ satisfies a twisted version of the local bisection hypothesis from \cite[Definition 4.9]{Ste19}, which says that every normaliser is supported on an open bisection.
Under our representation from \Cref{SteinbergAlgebraRepresentation}, $I$ identifies with the characteristic functions $\chi_{O_b}$ of the compact open bisections $O_b$ of the ultrafilter groupoid $\mathcal{U}(I)$ for $b\in I$.
Thus, we will show $D$ is quasi-Cartan under the assumption that

\begin{center}
    \textbf{$N$ satisfies $N\subseteq DI$,}
\end{center}

\noindent
which is analogous to the local bisection hypothesis.

\begin{proposition}\label{QuasiCartan}
    The subalgebra $D$ is quasi-Cartan.
\end{proposition}

\begin{proof}
    By \Cref{PseudoCartan}, $D$ is pseudo-Cartan.
    For any $n\in N\subseteq DI$, there are $d\in D$ and $b\in I$ such that $n=db$.
    We can write $d$ as a linear combination $d=\sum_{i\leq n}k_ie_i$ of elements of $E$.
    For each $i\leq n$, $e_i\Phi(b)=\Phi(e_ib)$ as in the proof of \Cref{PseudoCartan}.
    Also, $\Phi(b)=b\Phi(b)=\Phi(b)b$ because $\Phi(b)\leq b$.
    Thus, we can compute 
    \[
        \Phi(db)
        =\sum_{i\leq n}k_ie_i\Phi(b)
        =d\Phi(b)
        =db\Phi(b)
    \]
    and 
    \[
        \Phi(b)db
        =\sum_{i\leq n}k_i\Phi(b)e_ib
        =\sum_{i\leq n}k_ie_i\Phi(b)b
        =\sum_{i\leq n}k_ie_i\Phi(b),
    \]
    so $\Phi(n)=n\Phi(b)=\Phi(b)n$, as required.
\end{proof}


Twisted Steinberg algebras of discrete twists over Hausdorff ample groupoids were introduced in \cite{ACCLMR22}.
Because $D$ is quasi-Cartan (\Cref{QuasiCartan}), $A$ is isomorphic to the twisted Steinberg algebra of a discrete twist $\Sigma$ over a Hausdorff ample groupoid $G$ by \cite[Theorem 6.6]{9a23}.
The groupoid $G$ is used to characterise when $D$ is Cartan or diagonal \cite[Proposition 7.1]{9a23}.
Here, we show $G$ is isomorphic to $\UU(I)$ (\Cref{GpdIsom}).
By combining \cite[Proposition 7.1]{9a23} with aspects of non-commutative Stone duality \cite{Lawson2023}, we characterise when $D$ is Cartan or diagonal in terms of $I$ (\Cref{CartanDiagonal}).

Like in \cite[\S5]{9a23}, we write $\Sigma$ for $\mathcal{U}(N)$.
For each $b\in I$ and $n\in N$, let 
\[O_b(I)\coloneqq\{U\in\mathcal{U}(I)\mathrel{|}b\in U\}\quad\text{and}\quad O_n(N)\coloneqq\{U\in\Sigma\mathrel{|}n\in U\}.\]
Also, let $\leq_N$ denote the order relation on $N$.
First we identify $\mathcal{U}(I)$ with a wide open subgroupoid of $\Sigma$.

\begin{proposition}\label{IsomToTwist}
    The map $\theta:U\mapsto U^{\leq_N}$ is an isomorphism of $\mathcal{U}(I)$ onto 
    \[\Sigma_I\coloneqq\bigcup_{b\in I}O_b(N),\] 
    which is a wide open subgroupoid of $\Sigma$.
\end{proposition}

\begin{proof}
    For all $U\in\mathcal{U}(I)$ and $b\in U$, observe $U^{\leq_N}=(U\cap b^\geq)^{\leq_N}$ by \eqref{Directed}.
    Thus, for all $b\in U$, $\theta$ restricts to the homeomorphism $U\mapsto(U\cap b^\geq)^{\leq_N}$ from $O_b(I)$ to $O_b(N)$ attained by composing the homeomorphisms $V\mapsto V\cap b^\geq$ from $O_b(I)$ to $\mathcal{U}(b^\geq)$ and $W\mapsto W^{\leq_N}$ from $\mathcal{U}(b^\geq)$ to $O_b(N)$ from \Cref{PrincipalRestriction}.
    Therefore, $\theta$ is an open continuous map from $\mathcal{U}(I)$ to $\Sigma$.
    Applying \Cref{PrincipalRestriction} again in the other direction, if $V\in O_b(N)$ for some $b\in I$, then $V=\theta(U)$ for some $U\in O_b(I)$.
    That is, $\theta(\mathcal{U}(I))=\bigcup_{b\in I}O_b(N)=\Sigma_I$.
    Also, $\theta$ is injective.
    Indeed, if $U,V\in\mathcal{U}(I)$ are distinct, there is $b\in U\setminus V$.
    Because $I$ is a meet-semilattice and $b\notin V$, \Cref{UltrafilterComplements} implies $b\wedge c=0$ for some $c\in V$, so $O_b(N)\cap O_c(N)=\emptyset$ because $I^{\geq_N}\subseteq I$.
    Since $\theta(U)\in O_b(N)$ and $\theta(V)\in O_c(N)$, we must have $\theta(U)\neq\theta(V)$.
    Therefore, $\theta$ is a homeomorphism from $\mathcal{U}(I)$ to $\Sigma_I$.
    Moreover, because $E(I)=E(N)$ (\Cref{PseudoCartan}) and unit ultrafilters are precisely the ultrafilters containing an idempotent (\Cref{UltraUnit}), $\theta$ restricts to a homeomorphism from $\mathcal{U}(I)^0$ to $\Sigma^0$.

    For any $U,V\in\mathcal{U}(I)$, $0\in UV$ if and only if $0\in U^{\leq_N}V^{\leq_N}=\theta(U)\theta(V)$, so $U\cdot V$ is defined if and only if $\theta(U)\cdot\theta(V)$ is defined.
    In this case $\theta(U)\cdot\theta(V)=(\theta(U)\theta(V))^{\leq_N}=(U^{\leq_N}V^{\leq_N})^{\leq_N}$ and $\theta(U\cdot V)=(U\cdot V)^{\leq_N}=((UV)^\leq)^{\leq_N}$, which both equal $(UV)^{\leq_N}$.
    Also, $\theta(U^{-1})=(U^{-1})^{\leq_N}=(U^{\leq_N})^{-1}=\theta(U)^{-1}$.
    It follows that $\Sigma_I$ is a subgroupoid of $\Sigma$ and $\theta$ is an isomorphim of $\mathcal{U}(I)$ onto $\Sigma_I$.
    Note $\Sigma_I$ is open because $\Sigma_I=\bigcup_{b\in I}O_b(N)$.
    Since $\theta$ restricts to a homeomorphism from $\mathcal{U}(I)^0$ to $\Sigma^0$, we have $\Sigma_I^0=\Sigma^0$, so $\Sigma_I$ is wide.
\end{proof}

Following \cite[\S5]{9a23}, $G$ denotes the set of equivalences classes of $\Sigma$ with respect to
\[U\simeq V\quad\Leftrightarrow\quad U=kV\text{ for some }k\in K^\times,\]
where $K^\times$ is the group of units of $K$.
The associated quotient map $q:\Sigma\to G$ makes $G$ a groupoid under the product 
\[q(U)\cdot q(V)\coloneqq q(U\cdot V)\text{ defined precisely when }U\cdot V\text{ is defined}\]
and inversion 
\[q(U)^{-1}\coloneqq q(U^{-1}).\]
We recall various properties of $G$ from \cite{9a23}.
For all $U\in\Sigma$, $q(\mathsf{s}(U))=\mathsf{s}(q(U))$ and $q(\mathsf{r}(U))=\mathsf{r}(q(U))$.
The collection $\{q(O_n(N))\mathrel{|}n\in N\}$ is a basis of compact open \emph{bisection}s of $G$ (i.e.~subsets of $G$ on which both $\mathsf{s}$ and $\mathsf{r}$ are injective) for the quotient topology on $G$.
In particular, $G$ is ample.
Moreover, because $D$ is quasi-Cartan, $G$ is Hausdorff by \cite[Theorem 5.6]{9a23}.
The quotient map $q$ is an open continuous groupoid homomorphism, which identifies $\Sigma^0$ with $G^0$.

We aim to show $q$ restricts to an isomorphism of $\Sigma_I$ onto $G$.
We will draw on a number of lemmas from \cite{9a23}, as well as the following observations.

\begin{lemma}
    For all $k\in K$ and $b\in I$, 
    \begin{equation}\label{CharacteristicIdempotents}
        0\neq kb\in E\quad\Rightarrow\quad k=1\quad\text{and}\quad b\in E,
    \end{equation}
    \begin{equation}\label{RingUnits}
        0\neq kb\in N\quad\Rightarrow\quad k\in K^\times.
    \end{equation}
    For all $j,k\in K^\times$, $0\neq a\in I$ and $b\in I$,
    \begin{equation}\label{NormaliserOrder}
        ja\leq_Nkb\quad\Leftrightarrow\quad j=k\quad\text{and}\quad a\leq b.
    \end{equation}
\end{lemma}
\begin{proof}
    Fix $k\in K$ and $b\in I$.
    Suppose $0\neq kb\in E$.
    Note $k\Phi(b)=\Phi(kb)=kb\in E\setminus\{0\}$.
    Then, $k\Phi(b)=k\Phi(b)k\Phi(b)=k^2\Phi(b)$, so $k=k^2$ by \eqref{TorsionFree}, which means $k=1$ because $K$ is indecomposable.

    Now suppose $0\neq kb\in N$.
    Then, $(kb)^\dagger\in N\subseteq DI$, so $(kb)^\dagger=\sum_{i\leq n}k_ie_ic$, for $\{k_1,\ldots,k_n\}\subseteq K$, disjoint $\{e_1,\ldots,e_n\}\subseteq E$ and $c\in I$.
    There must be some $j\leq n$ such that $kbk_je_jc\neq 0$ -- otherwise $kb=kb(kb)^\dagger kb=0$.
    Because the $e_i$ are disjoint, multiplying $(kb)^\dagger$ on the left by $e_j$ shows $k_je_jc\leq_N(kb)^\dagger$, and so $kbk_je_jc\leq_Nkb(kb)^\dagger\in E$.
    Since $I$ commutes with $K$, we have $kk_jbe_jc=kbk_je_jc\in E$, and \eqref{CharacteristicIdempotents} implies $kk_j=1$, so $k\in K^\times$.

    Now fix $j,k\in K^\times$, $0\neq a\in I$ and $b\in I$.
    As per \cite[Lemma 2.16(b)]{9a23}, $\leq_N$ is preserved under left multiplication by $K^\times$.
    In particular, the ``$\Leftarrow$'' direction holds.
    Now, if $ja\leq_Nkb$, then we have $a\leq_Nj^{-1}kb$.
    Thus, $j^{-1}kba^{-1}=j^{-1}kba^{-1}aa^{-1}=aa^{-1}\in E$.
    Notice $0\neq j^{-1}kba^{-1}$ -- otherwise $a=aa^{-1}a=0a=0$, a contradiction.
    Hence, \eqref{CharacteristicIdempotents} implies $j^{-1}k=1$, so $j=k$ and $a\leq j^{-1}kb=1b=b$.
\end{proof}

\begin{lemma}\label{UltraBisectionScalars}
    For all $U\in\Sigma$ and $n\in U$, there are $k\in K^\times$ and $b\in I$ such that $n\geq_Nkb\in U$.
\end{lemma}

\begin{proof}
    Take any $n\in U$.
    Since $N\subseteq DI$, we can write $n=\sum_{i\leq m}k_if_ib$ for $\{k_1,\ldots,k_n\}\subseteq K$, disjoint $\{f_1,\ldots,f_n\}\subseteq E$ and $b\in I$.
    Applying \cite[Lemma 5.2]{9a23}, we compute 
    \[\sum_{i\leq m}f_in=\sum_{i,j\leq m}k_jf_if_jb=\sum_{i\leq m}k_if_ib=n,\]
    and conclude there is an $i$ such that $k_if_ib\in U$.
    In particular, $0\neq k_if_ib\in N$, and so $k_i\in K^\times$ by \eqref{RingUnits}.
    It remains to show $k_if_ib\leq_Nn$.
    Since $k_i\in K^\times$, \cite[Lemma 2.16]{9a23} tells us $(k_if_ib)^\dagger=k_i^{-1}(f_ib)^\dagger$.
    Since $f_ib\in I$, $k_i^{-1}(f_ib)^\dagger=k_i^{-1}(f_ib)^{-1}=k_i^{-1}b^{-1}f_i$.
    Because $I$ commutes with $K$, we have $\mathsf{s}(k_if_ib)=k_i^{-1}b^{-1}f_ik_if_ib=b^{-1}f_ib$.
    Thus, $n\mathsf{s}(k_if_ib)=\sum_{j\leq m}k_jf_jbb^{-1}f_ib=\sum_{j\leq m}k_jf_jf_ib=k_if_ib$, and so $k_if_ib\leq_Nn$.
\end{proof}

\begin{proposition}\label{IsomToG}
    The map $q$ restricts to an isomorphism of $\Sigma_I$ onto $G$.
\end{proposition}

\begin{proof}
    Recall $q$ is an open continuous groupoid homomorphism of $\Sigma$ onto $G$, which identifies $\Sigma^0$ with $G^0$, so it remains to show that $G\subseteq q(\Sigma_I)$ and $q$ is injective on $\Sigma_I$.
    Fix $U\in \Sigma$ and consider $q(U)\in G$.
    By \Cref{UltraBisectionScalars}, we have $kb\in U$ for some $k\in K^\times$ and $b\in I$.
    Then, $b=k^{-1}kb\in k^{-1}U$.
    The groupoid $\Sigma$ is closed under scalar multiplication by $K^\times$ (see \cite[Lemma 5.3]{9a23}), so $k^{-1}U\in\Sigma$.
    By definition of the equivalence relation, we get $q(U)=q(k^{-1}U)\in q(\Sigma_I)$, as needed.
    Now suppose $q(U)=q(V)$ for some $U,V\in\Sigma_I$, and find $k\in K^\times$ such $U=kV$.
    By definition of $\Sigma_I$, there are $b\in U\cap I$ and $c\in V\cap I$.
    Then, $b,kc\in U$ and \eqref{Directed} implies there is some $a\in U$ such that $a\leq_Nkb,c$.
    Since $a\leq_Nc\in I$, we also have $a\in I$.
    Also, $a\neq 0$ because $a\in U$.
    Thus, $1a=a\leq_Nkb$ implies $1=k$ by \eqref{NormaliserOrder}, and so $U=1V=V$.
\end{proof}

\begin{theorem}\label{GpdIsom}
    The map $U\mapsto q(U^{\leq_N})=q(\theta(U))$ is an isomorphism of $\mathcal{U}(I)$ onto $G$.
\end{theorem}

\begin{proof}
    Compose the isomorphisms from \Cref{IsomToTwist,IsomToG}.
\end{proof}

Building on results from \cite[Proposition 7.1]{9a23} and using non-commutative Stone duality \cite{Lawson2023}, we now use \Cref{GpdIsom} to characterise when $D$ is a Cartan or diagonal subalgebra of $A$ in terms of the inverse semigroup $I$ (\Cref{CartanDiagonal}).

\begin{corollary}
    \label{CartanDiagonal}
    If all compatible $a,b\in I$ are bounded, then
    \[D\text{ is Cartan}\quad\Leftrightarrow\quad\UU(I)\text{ is effective}\quad\Leftrightarrow\quad I\text{ is fundamental,}\]
    \[D\text{ is diagonal}\quad\Leftrightarrow\quad\UU(I)\text{ is principal}\quad\Leftrightarrow\quad I\text{ is basic.}\]
\end{corollary}

Lastly, we show the discrete twist $\Sigma$ over $G$ is isomorphic to the trivial discrete twist $K^\times\times\mathcal{U}(I)$ over $\mathcal{U}(I)$.
The set $K^\times\times\mathcal{U}(I)$, endowed with the product topology with respect to the discrete topology on $K^\times$, is a topological groupoid under the product
\[(j,U)\cdot(k,V)\coloneqq(jk,U\cdot V)\text{ defined precisely when }U\cdot V\text{ is defined}\]
and inversion 
\[(j,U)^{-1}\coloneqq(j^{-1},U^{-1}).\]
Building on the isomorphism $\theta$ of $\mathcal{U}(I)$ onto $\Sigma_I$ from \Cref{IsomToTwist}, we show $K^\times\times\mathcal{U}(I)$ is isomorphic to $\Sigma$.

\begin{proposition}\label{TwistIso}
    The map $(k,U)\mapsto k\theta(U)$ is an isomorphism of $K^\times\times\mathcal{U}(I)$ onto $\Sigma$.
\end{proposition}

\begin{proof}
    To see the map is injective, suppose $j\theta(U)=k\theta(V)$ for $j,k\in K^\times$ and $U,V\in\mathcal{U}(I)$.
    Taking any $b\in U$, we have $jb\in j\theta(U)=k\theta(V)$, so $jb=kn$ for some $n\in\theta(V)$.
    Find $c\in V$ such that $c\leq_Nn$.
    Then, $kc\leq_Nkn=jb$.
    By \eqref{NormaliserOrder}, $k=j$, so $\theta(U)=\theta(V)$ and hence $U=V$.
    For surjectivity, suppose $V\in\Sigma$.
    By \Cref{UltraBisectionScalars}, there are $j\in K^\times$ and $b\in I$ such that $kb\in V$.
    Then, $b=k^{-1}kb\in k^{-1}V\in\Sigma_I$, so $k^{-1}V=\theta(U)$ for some $U\in\mathcal{U}(I)$, and $(k,U)\mapsto k\theta(U)=kk^{-1}V=V$, as needed.
    Because $E(I)=E(N)$, the map identifies the units of $K^\times\times\mathcal{U}(I)$ with the units of $\Sigma$.
    It follows from various properties in \cite[Lemma 5.3]{9a23} that the map also identifies composable pairs with composable pairs, preserves the product and preserves the inversion.
    
    The map identifies the basic open sets $\{j\}\times O_b(I)$ with the open sets $O_{jb}(N)\subseteq\Sigma$, for $k\in K^\times$ and $b\in I$.
    Moreover, the collection of $O_{jb}(N)$ for $j\in K^\times$ and $b\in I$ form a basis for the topology on $\Sigma$.
    Indeed, if $U\in O_n(N)$ for a given $n\in N$, \Cref{UltraBisectionScalars} tells us there are $j\in K^\times$ and $b\in I$ such that $n\geq_Nkb\in U$.
    Then, $U\in O_{kb}(N)\subseteq O_n(N)$ by \eqref{Upset}.
    Hence, the map is open and continuous.
\end{proof}

Since $q(k\theta(U))=q(\theta(U))$ for all $k\in K^\times$ and $U\in\mathcal{U(I)}$, the following diagram commutes.
\begin{center}
    \begin{tikzcd}[row sep=1cm, column sep=2cm]
    K^\times\times\mathcal{U}(I) \arrow[r, "{(k,U)\mapsto U}"] \arrow[d, "{(k,U)\mapsto k\theta(U)}"'] & \mathcal{U}(I) \arrow[d, "{U\mapsto q(\theta(U))}"] \\
    \Sigma \arrow[r, "q"]                                                          & G                               
    \end{tikzcd}
\end{center}
Moreover, the actions $j\cdot(k,U)\coloneqq(jk,U)$ and $j\cdot V\coloneqq jV$ of $K^\times$ on $K^\times\times\mathcal{U}(I)$ and $\Sigma$, respectively, are preserved by the isomorphism from \Cref{TwistIso}.
That is, $(jk)\theta(U)=j(k\theta(U))$, for all $j,k\in K^\times$ and $U\in\mathcal{U}(I)$.
Therefore, the discrete twist $\Sigma$ over $G$ is isomorphic to the trivial discrete twist $K^\times\times\mathcal{U}(I)$ over $\mathcal{U}(I)$.

\newcommand{\etalchar}[1]{$^{#1}$}

\end{document}